\documentclass[a4paper, 12pt]{article}
\pdfoutput=1

\usepackage[margin=25mm]{geometry}

\usepackage[utf8]{inputenc}
\usepackage[T1]{fontenc}
\usepackage[british]{babel}
\usepackage[style=british]{csquotes}
\usepackage[en-GB]{datetime2}
\DTMlangsetup[en-GB]{ord=omit}
\usepackage{xcolor}

\usepackage{amsthm, thmtools}
\usepackage{mathtools}
\usepackage{titlesec} 
\usepackage{hyperref}
\hypersetup{
    bookmarksnumbered,
    colorlinks,
    linkcolor={cyan!50!blue},
    citecolor={cyan!50!blue},
    urlcolor={cyan!50!blue}
}
\usepackage[capitalise]{cleveref}

\usepackage{tikz}
\usepackage{tikz-cd}

\usepackage[expansion=false,nopatch=footnote]{microtype}
\usepackage[lining,tabular,scale=.9]{FiraSans}
\usepackage[tt=false,semibold]{libertinus}
\usepackage[libertine,smallerops]{newtxmath}
\usepackage[lining,scaled=.8]{FiraMono}
\usepackage[cal=boondox,frak=euler]{mathalpha}

\DeclareFontEncoding{MDA}{}{}
\DeclareSymbolFont{mathdesignA}{MDA}{mdput}{m}{n}
\SetSymbolFont{mathdesignA}{bold}{MDA}{mdput}{b}{n}
\DeclareSymbolFontAlphabet{\mathbb}{mathdesignA}

\tikzcdset{
    arrow style=tikz,
    arrows={/tikz/line width=.5pt},
    diagrams={>={Straight Barb[scale=0.8]}}
}

\DeclareFontFamily{OMX}{MnSymbolE}{}
\DeclareSymbolFont{MnLargeSymbols}{OMX}{MnSymbolE}{m}{n}
\SetSymbolFont{MnLargeSymbols}{bold}{OMX}{MnSymbolE}{b}{n}
\DeclareFontShape{OMX}{MnSymbolE}{m}{n}{
    <-6>  MnSymbolE5
   <6-7>  MnSymbolE6
   <7-8>  MnSymbolE7
   <8-9>  MnSymbolE8
   <9-10> MnSymbolE9
  <10-12> MnSymbolE10
  <12->   MnSymbolE12
}{}
\DeclareFontShape{OMX}{MnSymbolE}{b}{n}{
    <-6>  MnSymbolE-Bold5
   <6-7>  MnSymbolE-Bold6
   <7-8>  MnSymbolE-Bold7
   <8-9>  MnSymbolE-Bold8
   <9-10> MnSymbolE-Bold9
  <10-12> MnSymbolE-Bold10
  <12->   MnSymbolE-Bold12
}{}
\DeclareMathDelimiter{[}{\mathopen}{MnLargeSymbols}{'000}{MnLargeSymbols}{'000}
\DeclareMathDelimiter{]}{\mathclose}{MnLargeSymbols}{'005}{MnLargeSymbols}{'005}
\DeclareMathDelimiter{\llbr}{\mathopen}{MnLargeSymbols}{'102}{MnLargeSymbols}{'102}
\DeclareMathDelimiter{\rrbr}{\mathclose}{MnLargeSymbols}{'107}{MnLargeSymbols}{'107}

\usepackage{bm}
\usepackage{cases}
\usepackage{accents}

\usepackage{setspace}
\usepackage{needspace}

\usepackage{subdepth} 

\newcommand{\initlengths}{%
    \setlength{\abovedisplayshortskip}{3pt plus 9pt minus 3pt}%
    \setlength{\belowdisplayshortskip}{9pt plus 9pt minus 9pt}%
    \setlength{\abovedisplayskip}{9pt plus 9pt minus 9pt}%
    \setlength{\belowdisplayskip}{9pt plus 9pt minus 9pt}%
    \hfuzz 1pt%
    \tolerance 400
}

\newcommand{\authorinforule}{\noindent\rule{0.38\textwidth}{0.4pt}}

\newlength{\authorwidth}
\newcommand{\authorinfo}[3]{{%
    \raggedright
    \setlength{\leftskip}{1.5em}
    \setlength{\parindent}{0em}
    \setstretch{1}
    \par%
    {\small%
    \makebox[\authorwidth][l]{#1}%
    \texttt{#2}%
    \\
    #3.}
    \vspace{6pt}\par
}}

\usepackage{titling}
\pretitle{\vspace{0pt}\setstretch{1.05}\begin{center}\LARGE\libertinusDisplay\fontdimen2\font=0.3em} 
\posttitle{\end{center}\vspace{6pt}}

\newcommand{\parasep}{9pt plus 3pt minus 3pt}
\titleformat{\section}{\Large\libertinusDisplay\setstretch{1.1}}{\thesection}{1em}{}
\titleformat{\subsection}{\large\firalining}{\thesubsection}{1em}{}
\titleformat{\subsubsection}{\firamedium\boldmath}{\thesubsubsection}{1em}{}

\usepackage{tocloft}
\renewenvironment{abstract}{%
    \centering\begin{minipage}{.85\textwidth}%
    \setlength{\parindent}{1.5em}%
    \centerline{\large\firamedium\abstractname}%
    \par\vspace{12pt}%
}{\end{minipage}\par\vspace{3pt}}

\makeatletter
\declaretheoremstyle[
    spaceabove=\parasep, spacebelow=\parasep,
    postheadspace=.5em,
    headfont=\normalfont\bfseries,
    headpunct={},
    headformat={\NAME\ \NUMBER\NOTE.},
    notefont=\normalfont,
    notebraces={\ (}{)},
    bodyfont=\itshape,
]{theorem}
\declaretheoremstyle[
    spaceabove=\parasep, spacebelow=\parasep,
    postheadspace=.5em,
    headfont=\normalfont\bfseries,
    headpunct={},
    headformat={\NAME\NOTE.},
    notefont=\normalfont,
    notebraces={\ (}{)},
    bodyfont=\itshape,
]{theorem*}
\declaretheoremstyle[
    spaceabove=\parasep, spacebelow=\parasep,
    postheadspace=.5em,
    headfont=\normalfont\bfseries,
    headpunct={},
    headformat={\NAME\ \NUMBER\NOTE.},
    notefont=\normalfont,
    notebraces={\ (}{)},
]{definition}
\declaretheoremstyle[
    spaceabove=\parasep, spacebelow=\parasep,
    postheadspace=.5em,
    headfont=\normalfont\bfseries,
    headpunct={},
    headformat={\NAME\NOTE.},
    notefont=\normalfont,
    notebraces={\ (}{)},
]{definition*}
\declaretheoremstyle[
    spaceabove=\parasep, spacebelow=\parasep,
    postheadspace=.5em,
    headfont=\normalfont\bfseries,
    headpunct={},
    headformat={\NOTE},
    notefont=\normalfont\bfseries\boldmath,
    notebraces={\hspace{-.33em}}{.},
]{para}

\renewenvironment{proof}[1][\proofname]{\par
    \pushQED{\qed}%
    \normalfont\trivlist
    \item[\hskip\labelsep\bfseries #1\@addpunct{.}]\ignorespaces
}{%
    \popQED\endtrivlist\@endpefalse
}

\expandafter\def\csname equation*@qed\endcsname{\equation@qed}
\makeatother

\declaretheorem[parent=section, style=theorem, name=Theorem]{theorem}

\declaretheorem[sibling=theorem, style=theorem, name=Proposition]{proposition}

\declaretheorem[numbered=no, style=theorem*, name=Theorem]{theorem*} \declaretheorem[numbered=no, style=theorem*, name=Lemma]{lemma*}
\declaretheorem[numbered=no, style=theorem*, name=Proposition]{proposition*}
\declaretheorem[sibling=theorem, style=definition, name=Definition]{definition}
\declaretheorem[sibling=theorem, style=definition, name=Remark]{remark}
\declaretheorem[sibling=theorem, style=definition, name=Example]{example}

\declaretheorem[sibling=theorem, style=definition, name=Problem]{problem}

\numberwithin{equation}{section}

\crefdefaultlabelformat{#2{\upshape#1}#3}

\crefformat{equation}{#2{\upshape(#1)}#3}
\crefrangeformat{equation}{{\upshape#3(#1)#4--#5(#2)#6}}
\crefmultiformat{equation}{#2{\upshape(#1)}#3}{ and #2{\upshape(#1)}#3}{, #2{\upshape(#1)}#3}{, and~#2{\upshape(#1)}#3}

\crefformat{enumi}{#2{\upshape#1}#3}
\crefrangeformat{enumi}{{\upshape#3#1#4--#5#2#6}}
\crefmultiformat{enumi}{#2{\upshape#1}#3}{ and #2{\upshape#1}#3}{, #2{\upshape#1}#3}{, and~#2{\upshape#1}#3}

\crefformat{section}{#2{\upshape\S#1}#3}
\crefrangeformat{section}{{\upshape#3\S\S#1#4--#5#2#6}}
\crefmultiformat{section}{{\upshape#2\S\S#1#3}}{ and {\upshape#2#1#3}}{, {\upshape#2#1#3}}{, and~{\upshape#2#1#3}}
\crefformat{subsection}{#2{\upshape\S#1}#3}
\crefrangeformat{subsection}{{\upshape#3\S\S#1#4--#5#2#6}}
\crefmultiformat{subsection}{{\upshape#2\S\S#1#3}}{ and {\upshape#2#1#3}}{, {\upshape#2#1#3}}{, and~{\upshape#2#1#3}}
\crefformat{para}{#2{\upshape\S#1}#3}
\crefrangeformat{para}{{\upshape#3\S\S#1#4--#5#2#6}}
\crefmultiformat{para}{{\upshape#2\S\S#1#3}}{ and {\upshape#2#1#3}}{, {\upshape#2#1#3}}{, and~{\upshape#2#1#3}}

\crefname{figure}{Figure}{Figures}

\usepackage{enumitem}
\setlist{noitemsep}
\setlist[enumerate]{label=\textnormal{(\roman*)}}

\usepackage[labelsep=period, labelfont={bf}]{caption}

\usepackage[
    backend=biber,
    style=oxnum,
    giveninits,
    maxbibnames=5,
    maxcitenames=5,
    sorting=nyt
]{biblatex}

\DeclareFieldFormat*{title}{\textit{#1}}
\DeclareFieldFormat*{journaltitle}{#1}
\DeclareFieldFormat[article]{version}{\IfDecimal{#1}{version~}{}#1}
\DefineBibliographyStrings{english}{
  withafterword = {with an appendix by},
}

\DeclareFieldFormat{pages}{%
  \iffieldundef{bookpagination}%
    {#1}%
    {\mkpageprefix[bookpagination]{#1}}%
}

\newcommand{\PP}{\mathbb{P}}

\def\CC{\mathbb{C}}

\def\QQ{\mathbb{Q}}

\def\RR{\mathbb{R}}

\def\ZZ{\mathbb{Z}}

\def\sO{{\mathscr O}}

\renewcommand{\cal}{\mathcal}
\def\cA{{\cal A}}

\def\cD{{\cal D}}

\def\cG{{\cal G}}

\def\cM{{\cal M}}

\def\cP{{\cal P}}

\def\cX{{\cal X}}
\def\cY{{\cal Y}}

\def\cX{\mathcal{X} }

\def\bk{\mathbf{k}}

\def\tcX{\widetilde{\cX}}

\def\hbar{\overline{h}}

\def\PGL{\mathrm{PGL} }

\def\and{\quad{\rm and}\quad}
\def\lra{\longrightarrow }
\def\mapright#1{\,\smash{\mathop{\lra}\limits^{#1}}\,}

\def\beq{\begin{equation}}
\def\eeq{\end{equation}}
\def\ben{\begin{enumerate}}
\def\een{\end{enumerate}}

\def\DM{Deligne--Mumford }

\def\virt{^{\mathrm{vir}}}

\def\and{\quad\text{and}\quad}

\def\Fl{\mathrm{Fl}}

\def\Fl{\mathrm{Fl}}

\def\res{\mathsf{res}}

\def\T{\mathrm{T} }
\def\tmu{\tilde{\mu} }

\def\Fl{\mathrm{Fl} }

\def\PGL{\mathrm{PGL}}

\def\wcM{\widetilde{\cM}}

\newcommand{\git}{{/\mspace{-5mu}/}}

\newcommand{\sumbar}{\mathop{\mathrlap{\mathchoice{\mspace{4mu}\smash{\frac{\mspace{12mu}}{}}}{\mspace{1mu}\smash{\frac{\mspace{11mu}}{}}}{\smash{\frac{\mspace{8mu}}{}}}{\smash{\frac{\mspace{8mu}}{}}}}{\sum}}}

\newcommand{\leftsubstack}[2][6em]{\substack{\makebox[#1][l]{\scriptsize$\begin{aligned}#2\end{aligned}$}}}

\renewcommand{\leq}{\leqslant}
\renewcommand{\emptyset}{\varnothing}

\makeatletter
\def\big#1{{\hbox{$\left#1\vbox to10\p@{}\right.\n@space$}}}
\makeatother

\title{Generalized intersection pairings on\\moduli spaces of vector bundles over a curve}
\author{Chenjing Bu\quad and\quad Young-Hoon Kiem}
\date{}

\begin{document}

\initlengths

\maketitle
\vspace{-1em}

\begin{abstract}

We introduce the notion of a generalized intersection pairing for an Artin stack with a proper good moduli space and nonempty stable part. For the moduli stack of semistable bundles over a smooth projective curve, there are four known constructions by partial desingularization, parabolic bundles, stable pairs and wall crossing.
In this paper, we compare all these generalized intersection pairings by establishing wall crossing formulas between them.
Explicit computations for low rank cases are included.

\end{abstract}

\vspace{2em}
{
    \hypersetup{linkcolor=black}
    \tableofcontents
}

\vspace{1em}
\section{Introduction}

Let $X$ be a smooth projective variety with an ample line bundle $\sO_X(1)$ and let $\cM_\tau^{\mathrm{rig}}$ be the (rigidified) moduli stack of all coherent sheaves of fixed topological type $\tau$, which is usually an Artin stack of infinite type. 
As we don't know how to integrate cohomology classes for an arbitrary Artin stack, for sheaf counting, we often confine ourselves to the open substack $\cM^{\mathrm{ss}}_\tau$ (resp.~$\cM^{\mathrm{s}}_\tau$) consisting of Gieseker semistable (resp.~stable) sheaves. See \cite{Huybrechts-Lehn} for basic information on the Gieseker stability and moduli spaces. 

Suppose $\cM_\tau^{\mathrm{ss}}=\cM^{\mathrm{s}}_\tau$. Then 
the moduli stack $\cM^{\mathrm{ss}}_\tau$ of semistable sheaves is a projective scheme by \cite{Simpson}. 
If furthermore $X$ is a curve, $\cM^{\mathrm{ss}}_\tau$ is smooth and its cohomology has been much studied since 1970s. We know its Betti numbers by \cite{Harder-Narasimhan, atiyah-bott-1983} and the intersection pairing by \cite{jeffrey-kirwan-1998-curves}, just to indicate a few monumental results. 
When $X$ is a surface (cf.~\cite{mochizuki-2009-donaldson}) or a Calabi--Yau 3-fold (cf.~\cite{Thomas}) and $\cM_\tau^{\mathrm{ss}}=\cM^{\mathrm{s}}_\tau$, the moduli stack admits a virtual fundamental class $[\cM^{\mathrm{ss}}_\tau]\virt$ against which cohomology classes $\xi\in \mathrm{H}^\bullet (\cM^{\mathrm{ss}}_\tau)$ may be integrated to give us virtual \emph{intersection pairings}. 
Modern sheaf counting invariants in algebraic geometry like the Donaldson and Donaldson--Thomas invariants are all defined in this manner and are quite useful not just in algebraic geometry but also in quantum field theory and representation theory as well as in combinatorics. 

When $\cM^{\mathrm{ss}}_\tau\ne \cM_\tau^{\mathrm{s}}$, the moduli stack $\cM_\tau^{\mathrm{ss}}$ is an Artin stack with a proper good moduli space 
\begin{equation*}
    \begin{tikzcd}[column sep={8em, between origins}, row sep={3.5em, between origins}]
        \cM_\tau^{\mathrm{ss}}
        \ar[r, "\text{good moduli}"]
        & M_{\tau} 
         \\
       \cM^{\mathrm{s}}_\tau 
        \ar[u, hookrightarrow]
        \ar[r, "\cong"]
        & M_\tau^{\mathrm{s}}
        \ar[u, hookrightarrow]
    \end{tikzcd}
\end{equation*}
where $M_\tau$ is the geometric invariant theory quotient $Q \git G$ of a suitable Quot scheme by a reductive group action (cf.~\cite{simpson-1994}) while $\cM_\tau^{\mathrm{ss}}$ is the quotient stack $[Q^{\mathrm{ss}}/G]$ of the semistable part in $Q$. 
In contrast to the spectacular developments in the case of $\cM^{\mathrm{ss}}_\tau=\cM^{\mathrm{s}}_\tau$, not much has been known for sheaf counting when there are strictly semistable sheaves. 
On the other hand, even if one is only interested in the special case where no strictly semistable sheaves are allowed, it is inevitable to consider the general case because they appear naturally in wall crossing formulas as we vary the stability conditions. 

The purpose of this paper is to investigate \emph{generalized intersection pairings} on the moduli stack $\cM^{\mathrm{ss}}_\tau$ for any $\tau$ 
when $X$ is a curve.  
Here a generalized intersection pairing refers to a linear map
$$\mathrm{H}^\bullet (\cM^{\mathrm{ss}}_\tau)\lra \QQ$$
which restricts to the ordinary integral
\beq\label{b0}\int_{[\cM^{\mathrm{s}}_\tau]}:\mathrm{H}^\bullet_{\mathrm{c}}(\cM^{\mathrm{s}}_\tau)\lra \QQ \eeq
against the fundamental class of the smooth quasi-projective scheme $\cM^{\mathrm{s}}_\tau$. When $X$ is a surface or a Calabi--Yau/Fano 3-fold, 
$\cM_\tau^{\mathrm{s}}$ is quasi-smooth and we can consider the generalized virtual intersection pairing by using the virtual fundamental class $[\cM^{\mathrm{s}}_\tau]\virt$.

How do we get a generalized intersection pairing of an Artin stack $\cM^{\mathrm{ss}}_\tau$? 
To make sense of an integral, we have to resolve the stacky points with infinite stabilizers. In other words, 
we have to find a dominant morphism
$$f:Y\lra \cM^{\mathrm{ss}}_\tau$$
from a (quasi-smooth) proper \DM stack $Y$. 
If $f$ is an isomorphism over $\cM^{\mathrm{s}}_\tau$, we have a generalized intersection pairing
\beq\label{b1} \mathrm{H}^\bullet(\cM^{\mathrm{ss}}_\tau)\lra \QQ \ , \qquad \xi \longmapsto \int_{Y} \xi|_Y\eeq 
where $\xi|_Y=f^*\xi$ denotes the pullback by $f$. 
If $f$ is smooth over $\cM^{\mathrm{s}}_\tau$ with proper connected fibres, we also have a generalized intersection pairing 
\beq\label{b2}
\mathrm{H}^\bullet(\cM^{\mathrm{ss}}_\tau)\lra \QQ \ , \qquad \xi \longmapsto \int_{Y} \xi|_Y\cup \frac{e(\T_f)}{\chi_f}
\eeq
where $\T_f$ denotes the relative tangent bundle and $\chi_f$ denotes the topological Euler characteristic of fibres over $\cM^{\mathrm{s}}_\tau$. It is straightforward to see that both \cref{b1} and \cref{b2} restrict to \cref{b0}. 

Historically, the first generalized intersection pairing was constructed by Kirwan in \cite{kirwan-1985-partial} where she constructed a resolution
$$\rho:\widetilde{\cM}^{\mathrm{ss}}_\tau=[\widetilde{Q}^{\mathrm{s}}/G]\lra [Q^{\mathrm{ss}}/G]=\cM^{\mathrm{ss}}_\tau$$
of points with infinite stabilizers by a canonical sequence of blowups, so that $\widetilde{\cM}^{\mathrm{ss}}_\tau$ is proper Deligne--Mumford
when $X$ is a curve.
Kirwan's partial desingularization $\widetilde{\cM}^{\mathrm{ss}}_\tau$
gives us a generalized intersection pairing by
\beq\label{i1} 
\mathrm{H}^\bullet (\cM^{\mathrm{ss}}_\tau)\lra \QQ \ , \qquad \xi \longmapsto \int_{[\widetilde{\cM}^{\mathrm{ss}}_\tau]}\rho^*\xi.\eeq
This construction was generalized to higher dimensional sheaf counting by intrinsic blowups and semi-perfect obstruction theory in \cite{kiem-li-savvas-2024}.

The second generalized intersection pairing was introduced in \cite{jeffrey-kiem-kirwan-woolf-2003-cohomology,jeffrey-kiem-kirwan-woolf-2006-intersection} through an effort to compute \cref{i1}. 
We fix a point $p$ in a smooth projective curve $X$ and then consider parabolic bundles $(E,F)$ where $F$ is a full flag in the fibre $E|_p$ of $E$ over $p$. 
When the parabolic weight $c$ is close to zero and general, the moduli stack $\cP^{\mathrm{ss}}_\tau(c)$ of stable parabolic bundles is a smooth projective variety and the forgetful morphism 
$$\omega:\cP^{\mathrm{ss}}_\tau(c) \lra \cM^{\mathrm{ss}}_\tau \ , \qquad (E,F) \longmapsto E$$
is smooth. Furthermore, the Euler class of the relative tangent bundle $\T_\omega$ is the product $\cD$ of positive roots of $\mathrm{SL}_r$ and the topological Euler characteristic of the flag variety $\mathrm{Fl}(r)$ is $r!$ where $r$ denotes the rank of a vector bundle in $\cM^{\mathrm{ss}}_\tau$. 
We thus obtain a generalized intersection pairing
\beq\label{i2} 
\mathrm{JK}_\tau(c) :\mathrm{H}^\bullet (\cM^{\mathrm{ss}}_\tau)\lra \QQ \ , \qquad \xi \longmapsto \int_{[\cP^{\mathrm{ss}}_\tau(c)]}\omega^*\xi\cup \frac{\cD}{r!}\eeq
when $X$ is a curve. 
For $\dim X\ge 2$, we may use the Seidel--Thomas twists \cite{seidel-thomas-2001} and assume that $\cM^{\mathrm{ss}}_\tau$ consists of locally free sheaves. 
Then we may proceed as in the curve case to define a generalized virtual intersection pairing by \cref{i2}. 

The third generalized intersection pairing was introduced by T. Mochizuki in \cite{mochizuki-2009-donaldson} where he used the moduli stack $\cM'^{\mathrm{ss}}_\tau$ of $\epsilon$-stable pairs, with $\epsilon$ sufficiently close to 0, in order to resolve the stacky points in $\cM^{\mathrm{ss}}_\tau$. Here a pair refers to $(E,s)$ with $E\in \cM_\tau$ and $s\in \mathrm{H}^0(E)$. After twisting by a sufficiently ample line bundle if necessary, the forgetful morphism
$$\pi:\cM'^{\mathrm{ss}}_\tau\lra \cM^{\mathrm{ss}}_\tau$$
is smooth and the fibre over $E\in \cM^{\mathrm{s}}_\tau$ is $\PP^{\chi-1}$ where $\chi=\dim \mathrm{H}^0(E)$. Then we have a generalized intersection pairing
\beq\label{i3}
\Pi J'_\tau:\mathrm{H}^\bullet (\cM^{\mathrm{ss}}_\tau)\lra \QQ \ , \qquad \xi \longmapsto \int_{[\cM'^{\mathrm{ss}}_\tau]}\pi^*\xi\cup \frac{e(\T_\pi)}{\chi}\eeq
when $X$ is a curve. For higher dimensional sheaf counting when $\cM_\tau^{\mathrm{ss}}$ is quasi-smooth, we may use the virtual cycle $[\cM'^{\mathrm{ss}}_\tau]\virt$ in \cref{i3}. 

The last generalized intersection pairing that we deal with in this paper was constructed by Joyce in \cite{joyce-wall-crossing} where he showed that by a wall crossing \cite[Theorem 5.7]{joyce-wall-crossing} from the pairing \cref{i3}, we get a generalized intersection pairing 
\beq\label{i4}
J_\tau:
\mathrm{H}^\bullet(\cM^{\mathrm{ss}}_\tau)\lra \QQ\eeq
for $\dim X\le 2$,
which does not depend on the choice of the twisting in the construction of~$\cM'^{\mathrm{ss}}_\tau$. 

Now that we have at least four generalized intersection pairings,  
it seems natural to ask the following. 

\medskip

\noindent \textbf{Question}.
Can we compute and compare all these generalized intersection pairings \cref{i1}, \cref{i2}, \cref{i3} and \cref{i4}?  

\medskip

To be concrete, in the rest of this paper, we let $X$ be a curve and $\tau=(r,d)$ where $r$ and $d$ denote the rank and degree of a vector bundle. 
For curves, the generalized intersection pairing \cref{i2} was computed in \cite{jeffrey-kirwan-1998-curves} and \cref{i4} in \cite{bu-2023-curves}. 
The difference of \cref{i1} and \cref{i2} was computed in \cite{jeffrey-kiem-kirwan-woolf-2006-intersection} from which \cref{i1} follows by adding \cref{i2} (cf.~\S\ref{sec-par}).
The difference of \cref{i3} and \cref{i4} was expressed as a wall crossing formula in \cite[Theorem 5.7]{joyce-wall-crossing} from which \cref{i3} follows by adding \cref{i4}. The current state of knowledge may be summarized as follows:
\begin{equation}\label{i5}
    \begin{tikzcd}[column sep={6em, between origins}, row sep={2em, between origins}]
        \text{\cref{i1}}
        \ar[r, leftrightarrow, "\text{\cite{jeffrey-kiem-kirwan-woolf-2006-intersection}}"] 
        &
        \text{\cref{i2}}
        \ar[r, dotted, leftrightarrow, "\text{??}"]
        &
        \text{\cref{i3}}
        \ar[r, leftrightarrow, "\text{\cite{joyce-wall-crossing}}"] 
        &
        \text{\cref{i4}}
        \\
        &
        \text{\footnotesize\cite{jeffrey-kirwan-1998-curves}}
        &
        &
        \text{\footnotesize\cite{bu-2023-curves}}
    \end{tikzcd}
\end{equation}

\medskip

In this paper, we compute 
the difference of  \cref{i2} and \cref{i3} or \cref{i4}
by applying the machinery of \cite{joyce-wall-crossing}
where we think of the $\QQ$-linear maps $\mathrm{JK}_{r,d}(c)$ in \cref{i2}, $\Pi J'_{r,d}$ in \cref{i3} and $J_{r,d}$ in \cref{i4} as homology classes via the tautological duality between the homology and cohomology with $\QQ$-coefficients. 
By twisting $E$ by a sufficiently ample line bundle, we may assume that the slope $d/r$ of $E$ is large enough. 

\begin{theorem*}
    Let~$X$ be a smooth projective curve over $\mathbb{C}$
    of genus $g > 1$,
    and let $\mathcal{M}^{\smash{\mathrm{rig}}}_{k}$
    be the moduli stack of semistable vector bundles on~$X$
    of a fixed slope~$k \in \mathbb{Q}$.
    Consider Joyce's \textnormal{\cite{joyce-wall-crossing}}
    Lie bracket on the homology
    $\mathrm{H}_{\bullet + 2 \dim} (\mathcal{M}^{\smash{\mathrm{rig}}}_{k})$.
    Fix integers $r, d$ with $r > 0$ and $d/r = k$,
    and let $r_0 = r / \mathrm{gcd} (r, d)$.

\begin{enumerate}
\item \textnormal{(Theorem \ref{thm-wcf-general-weights})}
    For a generic parabolic weight $c$ close to $0$
    \textnormal{(cf.~Definition \ref{def-par-weights})},
    we have
    \begin{equation*}
        \mathrm{JK}_{r, d} (c) - J_{r, d} =
        \sum_{\substack{
            1^r = f_1 + \cdots + f_n: \\
            n > 1, \ r_0 \, | \, |f_i|
        }} {}
        \frac{r_1! \cdots r_n!}{r!} \cdot
        \widetilde{U} (f_1, \dotsc, f_n; \mu_0, \mu_c) \cdot
        [ [ \dotsc [ J_{r_1, d_1},
        J_{r_2, d_2} ], \dotsc ] ,
        J_{r_n, d_n} ] \ ,
    \end{equation*}
    where we sum over decompositions of the flag type
    $1^r = (1, \dotsc, 1)$
    into non-zero vectors
    $f_i = (f_{i, 1}, \dotsc, f_{i, r})$
    with $f_{i, j} \in \{0, 1\}$,
    we write $r_i = |f_i| = \sum_j f_{i, j}$
    and $d_i = k r_i$,
    and $\widetilde{U}$ are the universal combinatorial coefficients in \cref{def-wcf-coeff}. 
\item \textnormal{(Theorem \ref{thm-wcf-special-weights})} 
For special choices of $c$, the above formula can be significantly simplified as follows: Let 
\[
    c^+ = \bigl( \ c_i = 1 - \varepsilon^{i-1} \ \bigr) \ ,
    \qquad
    c^- = \bigl( \ c_i = \varepsilon^{r - i} \ \bigr) \ ,
    \qquad
    i = 1, \dotsc, r \ ,
\]
with $\varepsilon > 0$ sufficiently small. Then we have
    \begin{alignat*}{2}
        \mathrm{JK}_{r, d} (c^+) - J_{r, d}
        & =
        \sum_{\substack{
            r = r_1 + \cdots + r_n: \\
            n > 1, \ r_0 \, | \, r_i
        }} {}
        & \frac{(-1)^{n-1}}{n!} \cdot \frac{r_1}{r} \cdot
        [[ \dotsc [J_{r_1, d_1}, J_{r_2, d_2}], \dotsc, J_{r_n, d_n}]]
        & \ ,
        \\
        \mathrm{JK}_{r, d} (c^-) - J_{r, d}
        & =
        \sum_{\substack{
            r = r_1 + \cdots + r_n: \\
            n > 1, \ r_0 \, | \, r_i
        }} {}
        & \frac{r_1}{r} \cdot
        [[ \dotsc [J_{r_1, d_1}, J_{r_2, d_2}], \dotsc, J_{r_n, d_n}]]
        & \ ,
    \end{alignat*}
    where $r_i > 0$,
    and we write $d_i = k r_i$.
\item \textnormal{(Theorem \ref{thm7.5})} 
    Suppose that $k > 2g-2$.
    Then for a generic $c$ close to $0$, we have  
 \begin{align*}
        \mathrm{JK}_{r, d} (c) - \Pi J'_{r, d} =
        \hspace{-2.5em}
        \sum_{\substack{
            1^r = f_1 + \cdots + f_n: \\
            n > 1, \ r_0 \, | \, |f_i|, \\
            \mu_c (f_1) < \mu_c (f) < \mu_c (f_2) < \cdots < \mu_c (f_n)
        }} {}
        \hspace{-2.5em}
        \frac{r_1! \cdots r_n!}{r!} \cdot
        \frac{r_1}{r} \cdot
        [ [ \dotsc [ \Pi J'_{r_1, d_1},
        \mathrm{JK}_{r_2, d_2} (c) ], \dotsc ] ,
        \mathrm{JK}_{r_n, d_n} (c) ] \ ,
    \end{align*}
    where we write $r_i = |f_i| > 0$
    and $d_i = k r_i$,
    and $\mu_c$ is defined in \cref{def-par-weights}.
\end{enumerate}
\end{theorem*}

These formulas are deduced from applying
Joyce's formalism in \cite{joyce-wall-crossing}
to the moduli stack of triples $(E,F,s)$ over~$X$,
where $E$ is a semistable bundle of slope $k$,
$F$ is a flag in $E|_p$ and $s\in \mathrm{H}^0(E)$.
See \textnormal{\S\ref{subsec-joyce-axioms}} for the precise set-up.

\medskip

As summarized in \cref{i5}, together with the results in \cite{jeffrey-kirwan-1998-curves, bu-2023-curves, jeffrey-kiem-kirwan-woolf-2006-intersection, joyce-wall-crossing}, the above theorem enables us to compare all the generalized intersection pairings. 

For low rank cases, we can compute the generalized intersection pairings explicitly. 
When the rank is 2 and the degree is even, all the generalized intersection pairings in \cref{i5} are equal except for \cref{i1}. (See \S\ref{S8.1.1} and \S\ref{S8.2}.) When the rank is 3 or higher, we find that all the generalized intersection pairings are different in general. (See \S\ref{S8.1.2} and \S\ref{S8.1.3}.) 

\medskip

The layout of this paper is as follows. In \S\ref{S:GIP}, we introduce the notion of generalized intersection pairings and collect methods of construction. In \S\ref{S:npts}, we work out an example of GIT moduli space of $n$ points on the projective line $\PP^1$ where all the generalized intersection pairings can be explicitly computed and compared. In \S\ref{sec-moduli}, we collect all the moduli stacks used for generalized intersection pairings in vector bundle counting over curves. In \S\ref{sec-par}, we recall the comparison of \cref{i1} and \cref{i2} from \cite{jeffrey-kiem-kirwan-woolf-2006-intersection} for completeness. In \S\ref{sec-joyce}, we recall Joyce's wall crossing framework and his enumerative invariant from \cite{joyce-wall-crossing}. In \S\ref{sec-comparison}, we compare the generalized intersection pairings by parabolic bundles and by the Joyce class. In \S\ref{sec-computation}, we work out explicit computations for low rank cases.

\medskip

All the schemes and stacks in this paper are defined over $\CC$ and all the (co)homology groups have $\QQ$ coefficients. 

\medskip

\noindent
\textbf{Acknowledgements.}
We would like to thank Frances Kirwan for helpful discussions.
Chenjing Bu was supported by EPSRC grant reference EP/X040674/1.

\section{Generalized intersection pairings on Artin stacks}

\label{S:GIP}

For an Artin stack $\cX$ over $\mathbb{C}$ and a cohomology class $\xi\in \mathrm{H}^{\bullet}(\cX)$, it is not known how to make sense of the intersection pairing
$$\int_\cX \xi$$
in general, even when $\cX$ is smooth and admits a proper moduli space. 
In this section, we collect methods to define generalized intersection pairings on Artin stacks (Definition \ref{51}) by resolving the stacky points with infinite stabilizers or by wall crossings.

Let $\cX$ be an Artin stack over $\CC$ with a proper good moduli space $\cX\to X$. Suppose that 
the \DM substack $\cX^{\mathrm{s}}$ of stable points in $\cX$ is nonempty. Here a point $x\in\cX$ is called stable if $\pi^{-1}\pi(x)=\{x\}$ and the automorphism group of $x$ is finite. 
A typical example is the GIT quotient
\beq\label{1} \cX= [Q^{\mathrm{ss}}/G]\lra Q\git G=X\eeq
of a projective scheme $Q$ by a linearized action of a reductive group $G$ with nonempty stable part $Q^{\mathrm{s}}\ne \emptyset$. 

Let $\xi\in \mathrm{H}^{\bullet}(\cX;\QQ)$ be a cohomology class.
When $\cX$ is smooth Deligne--Mumford, we have the intersection pairing 
$$\int_{\cX}\xi \in \QQ$$
by using the fundamental class $[\cX]$ of $\cX$ (cf. \cite{Ful,Vis}). 
When $\cX$ is a proper \DM stack equipped with a virtual fundamental class $[\cX]\virt$ (cf. \cite{Li-Tian, Behrend-Fantechi}), we may consider the virtual intersection pairing 
$$\int_{[\cX]\virt}\xi$$
and many modern enumerative invariants including the Gromov--Witten and Donaldson--Thomas invariants are defined in this way.
 
For convenience, we introduce the following.
\begin{definition}\label{51}
    A \emph{generalized intersection pairing} on a smooth Artin stack $\cX$ over $\mathbb{C}$ of dimension~$d$ with a proper good moduli space is a linear map
    $$\int_{\cX}: \mathrm{H}^{2d}(\cX)\lra \QQ$$
    which restricts to the ordinary integral 
    $$\int_{\cX^{\mathrm{s}}} :\mathrm{H}^{2d}_{\mathrm{c}}(\cX^{\mathrm{s}})\lra \QQ$$    
of compactly supported cohomology classes on the stable part $\cX^{\mathrm{s}}$ by using the natural map $\mathrm{H}^{\bullet}_{\mathrm{c}}(\cX^{\mathrm{s}})\to \mathrm{H}^{\bullet}(\cX)$. 
\end{definition}

\begin{remark}
    (1) By the tautological duality $\mathrm{H}_\bullet(\cX)\cong \mathrm{Hom}(\mathrm{H}^\bullet(\cX),\QQ)$, we may consider a generalized intersection pairing as a homology class in $\mathrm{H}_{2d}(\cX)$ which restricts to the fundamental class of $\cX^{\mathrm{s}}$ in $\mathrm{H}^\mathrm{BM}_{\smash{2d}} (\mathcal{X}^{\mathrm{s}})$.

(2) When $\cX$ is quasi-smooth, if we replace $\cX^{\mathrm{s}}$ by its virtual fundamental class $[\cX^{\mathrm{s}}]\virt$ in Definition \ref{51}, we obtain the notion of a generalized virtual intersection pairing.   
\end{remark}

Down below, we collect known methods to define 
generalized intersection pairings for Artin stacks.

\subsection{Intersection pairings on partial desingularizations}\label{S1}

In \cite{Kir85}, a canonical partial desingularization 
\beq\label{2} \rho:\tcX=[\widetilde{Q}^{\mathrm{s}}/G]\lra [Q^{\mathrm{ss}}/G]= \cX \eeq
was constructed for the GIT quotient \cref{1} of a smooth projective variety $Q$ such that $\tcX$ is a smooth proper \DM stack.
Indeed, we blow up $Q^{\mathrm{ss}}$ along the locus of semistable points whose stabilizer group is of maximal dimension and then delete the unstable points with respect to the linearization which is very close to the pullback of the linearization of $Q$. By repeating this process finitely many times, we reach $\widetilde{Q}^{\mathrm{s}}=\widetilde{Q}^{\mathrm{ss}}$ in \cref{2}.

Using the partial desingularization \cref{2}, we have the first generalized intersection pairing 
\beq\label{3} \int_{\tcX} \xi|_{\tcX}\quad \text{for }\xi \in \mathrm{H}^{\bullet}(\cX)\eeq
where $\xi|_{\tcX}=\rho^*\xi$ denotes the pullback of $\xi$ by $\rho$. 
Of course, when there are no strictly semistable points so that $\cX=\cX^{\mathrm{s}}=[Q^{\mathrm{s}}/G]$ is Deligne--Mumford, $\tcX=\cX$ and \cref{3} coincides with the usual intersection pairing $\int_{\cX}\xi$ on the \DM stack $\cX$ in \cite{Ful, Vis}.

In \cite{JKKW1}, a systematic way to compute the generalized pairing \cref{3} was provided for  GIT quotients of smooth varieties. 
Under favorable circumstances, \cref{3} also gives us the intersection cohomology pairing on the good moduli space $X$ of $\cX$ by \cite{Kie1}. In \S\ref{S:npts}, we work out these intersection pairings explicitly for $(\PP^1)^{2r} \git \PGL_2$. 

\begin{remark}
The partial desingularization process \cref{2} was generalized to any Artin stack $\cX$ with a stable good moduli space by \cite{ER}. In \cite{KL, KLS}, a different generalization of \cref{2} was constructed by \emph{intrinsic blowups} which are better suited for lifting perfect obstruction theories when $\cX$ is quasi-smooth. 
In \cite{HRS}, this intrinsic partial desingularization was enhanced to the derived algebraic geometry setting.
When $\tcX$ admits a virtual fundamental class $[\tcX]\virt$ as in \cite{kiem-li-savvas-2024}, we may consider the generalized virtual intersection pairing
\beq \label{3v} \int_{[\tcX]\virt} \xi|_{\tcX}.\eeq
\end{remark}

The generalized intersection pairing in \cref{3} turned out to be quite useful. 
Using \cref{2} and \cref{3} from \cite{Kir85}, the intersection cohomology pairings on singular moduli spaces of semistable vector bundles over a smooth projective curve were computed in \cite{JKKW2} by using \cite{JKKW1, kiem-2004}. In \cite{KL}, the intrinsic partial desingularization and \cref{3v} were applied to prove a wall crossing formula for Donaldson--Thomas invariants. 
In \cite{KLS, KSk, HRS}, generalized Donaldson--Thomas invariants and related invariants were constructed by \cref{3v}.

\subsection{Intersection pairings on smooth fibrations}\label{S2}

Let $\cX$ be a connected Artin stack with a good moduli space $\cX\to X$ such that $\cX^{\mathrm{s}}\ne \emptyset$. Another way to define generalized intersection pairings of cohomology classes $\xi\in \mathrm{H}^{\bullet}(\cX)$ is to use a smooth morphism over $\cX$. 
Let \beq\label{6} \pi:\cY\lra \cX\eeq
be a \emph{smooth} morphism of Artin stacks such that $\cY$ is proper \DM and $\pi$ is proper over the stable part $\cX^{\mathrm{s}}$ with connected fibres. 

Let $w(\pi)$ denote the topological Euler characteristic of fibres over $\cX^{\mathrm{s}}$. 
Let $\T_\pi$ be the relative tangent bundle of $\pi$ and $e(\T_\pi)$ denote the top Chern class of $\T_\pi$. 
When $\cX=\cX^{\mathrm{s}}$ and $w(\pi)\ne 0$, we have the identity
\beq\label{7} \int_{\cX}\xi=\int_{\cY} \frac{e(\T_\pi)}{w(\pi)}\cup \xi|_{\cY}\eeq
by the projection formula. 
So we can define a generalized intersection pairing as 
\beq\label{8} \int_{\cY}  \frac{e(\T_\pi)}{w(\pi)}\cup \xi|_{\cY} \quad \text{for } \xi\in \mathrm{H}^{\bullet}(\cX).\eeq

\begin{remark}
Note that \cref{7} was used in \cite{Mar} in the form of 
\beq\label{9} \int_{\mu^{-1}(0)/K}\xi =\int_{\mu^{-1}(0)/T} \frac{\cD}{|W|}\cup \xi|_{\mu^{-1}(0)/T}
\eeq
where $\mu:Q\to \bk^*$ is the moment map for a Hamiltonian action of a compact Lie group $K$ with maximal torus $T$, $W$ is the Weyl group of $K$ and $\cD$ is the product of positive roots.
In \cite{jeffrey-kirwan-1998-curves}, Jeffrey and Kirwan 
used \cref{9} to compute the intersection pairings 
when $\cX$ is a moduli space of stable bundles over a curve whose rank and degree are coprime.  

In \cite{Moch, JS}, \cref{8} was adopted for a generalized theory of Donaldson and Donaldson--Thomas invariants where $\cX$ is a moduli stack of semistable sheaves (or perfect complexes) and $\cY$ is a moduli stack of $0^+$-stable pairs. 
\end{remark}

There are many choices of smooth fibrations \cref{6} and we have to deal with the following comparison issue.
\begin{problem}\label{10}
Let $\pi':\cY'\to \cX$ be another such smooth morphism. 
Compute the difference
\beq\label{11a} \int_{\cY}  \frac{e(\T_\pi)}{w(\pi)}\cup \xi|_{\cY} - \int_{\cY'}  \frac{e(\T_{\pi'})}{w(\pi')}\cup \xi|_{\cY'} 
\quad \text{for } \xi\in \mathrm{H}^{\bullet}(\cX).\eeq
\end{problem} 

For instance, when $\cX$ is a moduli stack of semistable bundles over a smooth projective curve, we can think of $0^+$-stable pairs or $0^+$-stable parabolic bundles to construct smooth morphisms over $\cX$. In \S\ref{S:comparison}, we will compute their difference by a wall crossing formula. 

On the other hand, \cref{8} is often easier to compute than \cref{3} because we can vary the stability for $\cY$ and study wall crossings as in the cases of stable pairs or stable parabolic bundles. 
So it is reasonable to consider the following.
\begin{problem}\label{11}
Compute the difference  
\[\int_{\tcX} \xi|_{\tcX}  -\int_{\cY}  \frac{e(\T_\pi)}{w(\pi)}\cup \xi|_{\cY}\]
of \cref{3} and \cref{8}.\end{problem}
As we discussed, \cref{3} is canonical and \cref{8} is more suited for computation. 
Note that this is the method of computing the generalized intersection pairings in \cite{JKKW2} where $\cY$ is the moduli space of stable parabolic bundles with parabolic weights close to $0$.

\subsection{Enumerative invariants by wall crossings}
Let $\cX$ be a (rigidified) moduli stack of objects in (a full subcategory of) an abelian category, semistable with respect to a weak stability condition. Suppose a list of assumptions in  \cite[Assumptions 5.1-5.2]{joyce-wall-crossing} hold. Then by \cite[Theorem 5.7]{joyce-wall-crossing}, we have a homology class
$$[\cX]_{\mathrm{inv}}\in H_*(\cX)$$
which gives us a generalized intersection pairing
\beq\label{e29}
\int_{[\cX]_{\mathrm{inv}}}\xi\quad \text{for }\xi\in \mathrm{H}^{\bullet}(\cX)\eeq
using the tautological pairing $\mathrm{H}^{\bullet}(\cX)\otimes H_*(\cX)\to \QQ$. 

When $\cX$ is a moduli stack of semistable bundles over a curve, by \cite[Theorem 7.63]{joyce-wall-crossing}, there is a wall crossing formula \cref{eq-pi-pair} which compares \cref{e29} with \cref{8} where $\cY$ is the moduli stack of $0^+$-stable pairs $(E,s\in H^0(E))$. 

In this paper, we will compare the above generalized intersection pairings for the moduli spaces of semistable bundles over a smooth projective curve of genus $g\ge 2$.

\subsection{Intersection cohomology pairings}

It was proved in \cite[Theorem 5.3]{kiem-2004} that there is a canonical embedding $$\imath:\mathrm{IH}^{\bullet}(X)\lra \mathrm{H}^{\bullet}(\cX)$$
of the (middle perversity) intersection cohomology of the good moduli space $X$ into the cohomology of the stack $\cX$ when $\cX=[Q^{\mathrm{ss}}/G]$ is the GIT quotient stack of a smooth projective variety $Q$ by a linearized reductive group action with $Q^{\mathrm{s}}\ne \emptyset$ and the action is almost balanced (cf. \cite[Definition 5.1]{kiem-2004}).
For instance, all the moduli stacks of semistable vector bundles over a smooth projective curve satisfy this assumption (cf. \cite[Proposition 7.4]{kiem-2004}). 

The intersection cohomology $\mathrm{IH}^{\bullet}(X)$ does not have a ring structure but it has an intersection pairing which gives us the Poincar\'e duality. 
By \cite[Theorem 5.3]{kiem-2004}, the intersection cohomology pairing can be computed by
$$\langle \alpha,\beta\rangle_{\mathrm{IH}^{\bullet}(X)}=\int_{\widetilde{\cX}} \imath(\alpha)\cup \imath(\beta)|_{\widetilde{\cX}}$$
where $\tcX=[\widetilde{Q}^{\mathrm{s}}/G]$ denotes the partial desingularization in \S\ref{S1} for $\alpha, \beta\in \mathrm{IH}^{\bullet}(X)$ of complementary degrees. 

Considered as cycles on $X$, $\alpha$ and $\beta$ intersect at finitely many points on the stable part $X^{\mathrm{s}}=Q^{\mathrm{s}}/G$. 
As $\imath$ is defined by taking the inverse images of intersection homology cycles by the quotient map $$Q^{\mathrm{ss}}\lra Q \git G=X,$$
$\imath(\alpha)\cup \imath(\beta)$ is represented by a 0-cycle which consists of finitely many points in $\cX^{\mathrm{s}}$. 
In particular, $\imath(\alpha)\cup \imath(\beta)$ is a compactly supported cohomology class. 
By Definition \ref{51}, the following is a direct consequence of \cite[Theorem 5.3]{kiem-2004}.
\begin{proposition}\label{P2.6} The intersection cohomology pairing  can be computed by any generalized intersection pairing $\int_{\cX}$ by 
\beq\label{52}
\int_{\cX} \imath(\alpha)\cup \imath(\beta) = \langle \alpha,\beta\rangle_{\mathrm{IH}^{\bullet}(X)}
\eeq 
when $\cX=[Q^{\mathrm{ss}}/G]$ is the GIT quotient stack of a smooth projective variety with an almost balanced action. 
\end{proposition}

\section{GIT moduli space of \texorpdfstring{$n$}{n} points on projective line}

\label{S:npts}

In this section, we compute generalized intersection pairings on the GIT moduli space of $n$ points on $\PP^1$ up to the automorphism group $\mathrm{Aut}(\PP^1)$, as the first nontrivial example. 

Let $P_n=(\PP^1)^n$ with $n\ge 3$ and $G=\mathrm{SL}_2(\CC)$, $\overline G=\PGL_2(\CC)=\mathrm{Aut}(\PP^1)$. We denote the projection to the $j$th factor by $\pi_j:P_n\to \PP^1$ and let $T\cong \CC^*$ be the maximal torus in $G$. 
The diagonal action of $G$ on $P_n$ can be linearized by the line bundle $\sO(1,\cdots, 1)$ and we consider the semistable (resp.~stable) part $P_n^{\mathrm{ss}}$ (resp. $P_n^{\mathrm{s}}$) which consists of $n$-tuples of points in $\PP^1$ where at most $n/2$ (resp. less than $n/2$) points are allowed to coincide (cf. \cite{MFK}). Let
\beq\label{e1}\cY_n=[P_n^{\mathrm{ss}}/\overline G] \lra P_n \git \overline G=P_n\git G=Y_n\eeq
be the natural morphism from the quotient stack to the GIT quotient scheme. 

In this section, we compute the intersection pairings of cohomology classes in $\mathrm{H}^\bullet (\cY_n)$. 
By \cite{Kir84}, we have a surjective restriction map
\beq\label{e3} \mathrm{H}^\bullet _G(P_n)\lra \mathrm{H}^\bullet _G(P_n^{\mathrm{ss}})\cong \mathrm{H}^\bullet _{\overline G}(P_n^{\mathrm{ss}})\cong\mathrm{H}^\bullet (\cY_n),\quad \eta\mapsto \eta|_{P_n^{\mathrm{ss}}}\eeq 
and if we let $y_i=\pi_i^*c_1(\sO_{\PP^1}(1))$,  we have an isomorphism 
$$\mathrm{H}^\bullet _G(P_n)\cong\QQ[t^2,y_1,\cdots,y_n]/\langle t^2-y_1^2,\cdots, t^2-y_n^2\rangle $$
where $t\in \mathrm{H}^\bullet (BT)$ denotes the first Chern class of the universal line bundle over $BT$. Hence it suffices to compute the generalized intersection pairings on $\cY_n$ of 
\beq\label{e5}
\xi=t^{2m}y_1^{a_1}y_2^{a_2}\cdots y_n^{a_n}|_{P_n^{\mathrm{ss}}}\quad \text{with}\quad 2m+\sum_ia_i=n-3.\eeq 
We will often drop the restriction symbol $|_{P_n^{\mathrm{ss}}}$ to simplify the notation. 

\subsection{Intersection pairings for \texorpdfstring{$n$}{n} odd}\label{S6.1}
When $n$ is odd, $P_n^{\mathrm{ss}}=P_n^{\mathrm{s}}$ and \cref{e1} is an isomorphism. By \cite{JK0, GK, Mar}, the intersection pairngs
$$\int_{Y_n}\xi$$ 
can be computed as follows. 
Let $$\mu:P_n\lra \mathbf{k}^*$$ denote the moment map for the diagonal action of the compact group $K=\mathrm{SU}(2)$  
with $\mathbf{k}=\mathrm{Lie}(K)$ such that we have an inclusion $\mu^{-1}(0)\hookrightarrow P^{\mathrm{ss}}_n$ which induces the diffeomorphism
$$Y_n\cong \mu^{-1}(0)/\overline K,\quad \overline K=K/\{\pm 1\}=\mathrm{PU}(2).$$ 
Let $$\mu_T:P_n\mapright{\mu}\mathbf{k}^*\lra\RR$$
be the moment map for the action of the maximal torus $\mathrm{U}(1)\le K$ where the second arrow is the dual of the inclusion  $\RR=\mathrm{Lie(U}(1))\hookrightarrow  \mathrm{Lie}(K)=\mathbf{k}$. 
In explicit terms, if we identify $\PP^1$ with the unit sphere $S^2\subset \RR^3$, $\mu_T$ is the sum of the $z$-coordinates. 

As $0$ is a regular value of $\mu$, we can reduce the computation to the torus quotient by 
\beq\label{e2}
\int_{Y_n}\xi =\int_{\mu^{-1}(0)/\overline K} \xi= 2\int_{\mu^{-1}(0)/K}\xi=-4\int_{\mu_T^{-1}(0)/T} t^2\xi|_{\mu_T^{-1}(0)}\eeq
by Martin's trick in \cite{martin-1997-thesis}. 
The abelian quotient $\mu_T^{-1}(c)/T$ may change only at odd integers which we call walls. Note that $\mu_T^{-1}(c)=\emptyset$ for $c>n$. 

Now the difference of the intersection pairings as $c$ crosses a wall $c_0$ is computed by the residue formula (cf. \cite{GK})
\beq\label{e17} \int_{\mu_T^{-1}(c_+)/T} t^2\xi|_{\mu^{-1}(c_+)} -
\int_{\mu_T^{-1}(c_-)/T} t^2\xi|_{\mu^{-1}(c_-)} 
=\res_{t=0} \int_{F_{c_0}}\frac{t^2\xi|_{F_{c_0}}}{e(N_{F_{c_0}/P_n})}\eeq 
where $c_-<c_0<c_+$ are close to $c_0$ and $F_{c_0}$ denotes the $T$-fixed point locus in $\mu^{-1}_T(c_0)$. 
Therefore, we obtain the formula 
\beq\label{e4}
\int_{Y_n}\xi = 4\sum_{c>0} \res_{t=0} \int_{F_{c}}\frac{t^2\xi|_{F_{c}}}{e(N_{F_{c}/P_n})}.
\eeq

Let us denote the $T$-fixed points in $\PP^1$ by $p_0=0$ and $p_1=\infty$. With the identification $\PP^1=S^2$, $p_0$ corresponds to the south pole and $p_1$ to the north pole whose moment map values are $-1$ and $1$ respectively. 
Then the $T$-fixed point locus in $P_n$ consists of $2^n$ points $$\{p_\varepsilon=(p_{\varepsilon(1)},\cdots,p_{\varepsilon(n)})\in P_n\}$$
where $\varepsilon:\{1,\cdots, n\}\to\{0,1\}$ is a map. 
Let $$|\varepsilon|=\sum_i \varepsilon(i), \quad |a|=\sum_i a_i, \quad \varepsilon\cdot a=\sum_i a_i\varepsilon(i).$$
With this notation, we find that  $$\mu_T(p_\varepsilon)=\sum_{i=1}^n\mu_T(p_{\varepsilon(i)})=\sum_{i=1}^n (2\varepsilon(i)-1)=2|\varepsilon|-n.$$

Since $T$ acts on the fibre of $\sO_{\PP^1}(1)$ over $p_i$ with weight $(-1)^i$ for $i=0,1$, 
$y_i|_{p_\varepsilon}=(-1)^{\varepsilon(i)}t$ and the restriction of $\xi$ in \cref{e5} to $p_\varepsilon$ is
\beq\label{e6} \xi|_{p_\varepsilon}=(-1)^{\varepsilon\cdot a}t^{n-3}.\eeq

The normal bundle to $p_\varepsilon$ in $P_n$ is the direct sum 
$$\bigoplus_i \pi_j^*\sO_{\PP^1}(2)|_{p_{\varepsilon(j)}}$$
whose $T$-equivariant Euler class is 
\beq\label{e7}
e(N_{p_{\varepsilon}/P_n})=(2t)^{n-|\varepsilon|}(-2t)^{|\varepsilon|}=(-1)^{|\varepsilon|}2^nt^n.
\eeq

Combining \cref{e6} and \cref{e7}, we find that \cref{e4} is 
\beq\label{e8}
\int_{Y_n}\xi=\int_{Y_n} t^{2m}y_1^{a_1}y_2^{a_2}\cdots y_n^{a_n} =2^{2-n}\sum_{\varepsilon:|\varepsilon|>n/2} (-1)^{\varepsilon\cdot a+|\varepsilon|}.
\eeq
In particular, when $n=3$, $\int_{Y_3}1=1$ as $Y_3=\mathrm{pt}$.

\subsection{Generalized intersection pairings for \texorpdfstring{$n$}{n} even}
When $n=2r$ is even, $\cY_n=\cY_{2r}$ is an Artin stack for which we consider the generalized intersection pairings of 
\beq\label{e9}\xi=t^{2m}y_1^{a_1}y_2^{a_2}\cdots y_{2r}^{a_{2r}}\quad \text{with}\quad 2m+|a|=2r-3.\eeq
We note that $|a|$ is odd. 

\subsubsection{Intersection pairings by a smooth resolution}
\label{subsubsec-npts-smooth-resolution}

The most obvious smooth resolution of $\cY_{2r}=[P_{2r}^{\mathrm{ss}}/\overline G]$ is induced by the map $$P_{2r+1}=(\PP^1)^{2r+1}\lra (\PP^1)^{2r}=P_{2r}$$ which forgets the last factor. Since the stable part $P_{2r+1}^{\mathrm{s}}$ consists of $(2r+1)$-tuples where at most $r$ points are allowed to coincide, the image of $P^{\mathrm{s}}_{2r+1}$ by the forgetful map lies in the semistable part $P_{2r}^{\mathrm{ss}}$. 
Therefore, we have a smooth morphism 
\beq\label{e11}Y_{2r+1}=\cY_{2r+1}=[P_{2r+1}^{\mathrm{s}}/\overline G]\lra [P^{\mathrm{ss}}_{2r}/\overline G]=\cY_{2r}\eeq
which is a $\PP^1$-bundle over the stable part $\cY_{2r}^{\mathrm{s}}=[P_{2r}^{\mathrm{s}}/\overline G]=Y_{2r}^{\mathrm{s}}$.

The Euler class of the relative tangent bundle of \cref{e11} is $2y_{2r+1}$ and the topological Euler characteristic of the general fibre $\PP^1$ is $2$. Therefore, the generalized intersection pairing of \cref{e9} by the smooth resolution \cref{e11} is
\beq\label{e13}
\int_{Y_{2r+1}}\xi\cdot \frac{2y_{2r+1}}{2}=\int_{Y_{2r+1}} \xi\cdot y_{2r+1}. \eeq

By \cref{e8}, it is now straightforward to compute \cref{e13} with $\xi$ in \cref{e9}. Indeed, for $\varepsilon:\{1,2,\cdots, 2r\}\to \{0,1\}$, we have  
\beq\label{e10}
\int_{Y_{2r+1}}\xi\cdot y_{2r+1} = 
2^{1-2r}\sum_{\varepsilon:|\varepsilon|\ge r+1}(-1)^{\varepsilon\cdot a+|\varepsilon|} 
+ 2^{1-2r} \sum_{\varepsilon:|\varepsilon|\ge r}(-1)^{\varepsilon\cdot a+|\varepsilon|} 
\eeq
where the first term on the right hand side is the sum over $\varepsilon$ with the last point at $p_0=0\in \PP^1$ and the second term is over $\varepsilon$ with the last point at $p_1=\infty\in \PP^1$. 
When $|\varepsilon|=r$, $1-\varepsilon$ is also a map to $\{0,1\}$ with $|1-\varepsilon|=r$. Since 
$(1-\varepsilon)\cdot a+|1-\varepsilon|=|a|-\varepsilon\cdot a+2r-|\varepsilon|$ and $|a|$ is odd, we have
$$(-1)^{(1-\varepsilon)\cdot a+|1-\varepsilon|}=-(-1)^{\varepsilon\cdot a+|\varepsilon|}$$
and hence 
\beq\label{e15}\sum_{\varepsilon:|\varepsilon|= r}(-1)^{\varepsilon\cdot a+|\varepsilon|} =\sum_{\varepsilon:|\varepsilon|= r}(-1)^{\varepsilon\cdot a+r} =0\eeq
Therefore we find that the generalized intersection pairing \cref{e13} is 
\beq\label{e14}
\int_{Y_{2r+1}}\xi\cdot \frac{2y_{2r+1}}{2}=
2^{2-2r}\sum_{\varepsilon:|\varepsilon|\ge r+1}(-1)^{\varepsilon\cdot a+|\varepsilon|}.\eeq

\begin{remark}  We may consider $Y_{2r+1}$ as the non-reductive GIT quotient
\[ P_{2r} \git B=(P_{2r}\times (G/B)) \git G\]
where $B$ is a Borel subgroup of $G$ and the right hand side is the GIT quotient with respect to the ample line bundle $\sO_{P_{2r}}(1)\boxtimes \sO_{G/B}(\epsilon)$ for sufficiently small $\epsilon>0$. 
It is straightforward to see that the induced map 
$$P_{2r} \git B=(P_{2r}\times (G/B)) \git G\lra P_{2r} \git G$$
is the morphism of good moduli spaces induced from \cref{e11}.  
\end{remark}

\subsubsection{Intersection pairings by partial desingularization}

Now we compute the generalized intersection pairing of $\xi$ in \cref{e9} by using the partial desingularization $\widetilde{Y}_{2r}$ of $Y_{2r}$. 

As in \S\ref{S6.1}, for sufficiently small $\epsilon>0$ in $\mathrm{Lie}(T)^*=\RR$, we can compute 
\beq\label{e16}
2\int_{\mu^{-1}(\epsilon)/T}\xi\, \frac{2t}{2}=-4\int_{\mu_T^{-1}(\epsilon)/T} t^2\xi\eeq
by the resudue formula \cref{e17} as 
\beq\label{e20} 4\sum_{\varepsilon:|\varepsilon|> r}\res_{t=0}\frac{t^2\xi|_{p_\varepsilon}}{e(N_{p_\varepsilon/P_{2r}})}=2^{2-2r}\sum_{\varepsilon:|\varepsilon|\ge r+1} (-1)^{\varepsilon\cdot a+|\varepsilon|}.\eeq

Observe that the strictly semistable points with reductive stabilizer groups in $P_{2r}$ are precisely the $G$-orbits of $p_\varepsilon$ with $$\varepsilon:\{1,2,\cdots,2r\}\lra  \{0,1\},\quad  |\varepsilon|=r.$$ Hence we blow up $P_{2r}^{\mathrm{ss}}$ along $\sqcup_{|\varepsilon|=r}Gp_\varepsilon$ and delete unstable points to get $\widetilde{P}_{2r}^{\mathrm{s}}$. The partial desingularization is then the quotient 
$\widetilde{Y}_{2r} =[\widetilde{P}_{2r}^{\mathrm{s}}/\overline G]
=\widetilde{P}_{2r} \git G$ of $\widetilde{P}_{2r}^{\mathrm{s}}$.

The new $T$-fixed point set after the blowup that lies over the fixed point $p_\varepsilon$ with $|\varepsilon|=r$ is 
$\PP W_\varepsilon^+\sqcup \PP W^-_\varepsilon$ where $$N_{Gp_\varepsilon/P_{2r}}=W^+_\varepsilon\oplus W_\varepsilon^-$$ 
is the decomposition of the normal space $N_{Gp_\varepsilon/P_{2r}}$ to $Gp_\varepsilon$ into positive and negative weight parts of the induced $T$-action. 
Hence the only new wall crossing terms after the blowup are 
\beq\label{e19}
4\sum_{|\varepsilon|=r} \res_{t=0} \int_{\PP W_\varepsilon^+} \frac{t^2\xi|_{p_\varepsilon}}{e(N_{\PP W^+_\varepsilon/\widetilde{P}_{2r}})}.
\eeq
By \cref{e6}, $\xi|_{p_\varepsilon}=(-1)^{\varepsilon\cdot a}t^{2r-3}$. Since 
$$N_{\PP W^+_\varepsilon/\widetilde{P}_{2r}}=\sO_{\PP W_\varepsilon^+}(1)\otimes W_\varepsilon^-\oplus \sO_{\PP W_\varepsilon^+}(-1)\oplus N_{p_\varepsilon/Gp_\varepsilon},$$ we find that 
$$e(N_{\PP W^+_\varepsilon/\widetilde{P}_{2r}})=4t^2(y-2t)^r.$$
Hence, \cref{e19} is 
    $$4\sum_{|\varepsilon|=r} \res_{t=0} \int_{\PP W_\varepsilon^+} \frac{(-1)^{\varepsilon\cdot a}t^{2r-1}}{4t^2(y-2t)^r}
 =\sum_{|\varepsilon|=r} \res_{t=0} \int_{\PP^{r-2}} (-1)^{\varepsilon\cdot a+r}2^{-r}t^{r-3}\left(1-\frac{y}{2t}\right)^{-r}$$
    which equals 
    $$\sum_{|\varepsilon|=r} \res_{t=0} \int_{\PP^{r-2}} (-1)^{\varepsilon\cdot a+r}2^{-r}t^{r-3}\sum_k \binom{-r}{k}\left(\frac{-y}{2t}\right)^k$$
    $$=\sum_{|\varepsilon|=r}\int_{\PP^{r-2}} (-1)^{\varepsilon\cdot a+r}2^{-r} \binom{-r}{r-2}\frac{(-1)^{r-2}y^{r-2}}{2^{r-2}}.$$
    Since $\int_{\PP^{r-2}}y^{r-2}=1$ and $$\binom{-r}{r-2}=(-1)^{r-2}\binom{2r-3}{r-2},$$
    we find that \cref{e19} equals 
    $$2^{2-2r}\binom{2r-3}{r-2}\sum_{|\varepsilon|=r} (-1)^{\varepsilon\cdot a+r}$$
    which vanishes by \cref{e15}.

    By adding \cref{e20} and \cref{e19}, we find that the generalized intersection pairing of $\xi$ by the partial desingularization $\widetilde{Y}_{2r}$ is 
\beq\label{e21}
\int_{\widetilde{Y}_{2r}}\xi|_{\widetilde{Y}_{2r}} = 2^{2-2r}\sum_{\varepsilon:|\varepsilon|\ge r+1} (-1)^{\varepsilon\cdot a+|\varepsilon|}.
\eeq 

\subsubsection{A comparison theorem}
By comparing \cref{e14} with \cref{e21}, we find the equality of generalized intersection pairings by the smooth resolution $Y_{2r+1}=\cY_{2r+1}\to \cY_{2r}$ and by the partial desingularization $\widetilde{Y}_{2r}\to \cY_{2r}$. 
\begin{proposition}\label{P3.2}
For any $\xi\in \mathrm{H}^\bullet (\cY_{2r})$, we have the equality 
\beq\label{e22}
\int_{\widetilde{Y}_{2r}}\xi|_{\widetilde{Y}_{2r}} =\int_{Y_{2r+1}} \xi|_{Y_{2r+1}}\cup \frac{e(T_{\cY_{2r+1}/\cY_{2r}})}{\chi(\PP^1)}. 
\eeq
\end{proposition}

\subsubsection{Intersection cohomology pairings}\label{S3.2.4} 

The action of $G$ on $P_{2r}$ is weakly balanced (cf. \cite[Definition 7.2]{kiem-2004}) and hence we have an embedding 
$$\imath:\mathrm{IH}^\bullet (Y_{2r})\hookrightarrow \mathrm{H}^\bullet _G(P_{2r}^{\mathrm{ss}})=\mathrm{H}^\bullet (\cY_{2r})$$
whose image is  
\beq\label{e26}
\mathrm{IH}^\bullet (Y_{2r})\cong \imath(\mathrm{IH}^\bullet (Y_{2r}))=\{\xi\in \mathrm{H}^\bullet _G(P_{2r}^{\mathrm{ss}})\,|\, \xi|_{p_\varepsilon}\in H^{\le 2r-4}_T\ \ \forall \varepsilon \text{ with }|\varepsilon|=r\}\eeq
by \cite[Theorem 7.7]{Kie1}. 
In particular, the Poincar\'e polynomial of $\mathrm{IH}^\bullet (Y_{2r})$ is
$$\mathrm{IP}_t(Y_{2r})=\sum_j z^j\dim \mathrm{IH}^j(Y_{2r})=\frac{(1+z^2)^{2r}}{1-z^4}-\frac12 \binom{2r}{r} \frac{z^{2r-2}}{1-z^2}$$
because the restriction map $\mathrm{H}^{\ge 2r-2}_G(P_{2r}^{\mathrm{ss}})\to \left(\bigoplus_{|\varepsilon|=r}\mathrm{H}^{\ge 2r-2}_T(p_\varepsilon)\right)^{\ZZ_2}$ is surjective by \cite[Lemma 9.2]{kiem-2004}. 
Moreover, by \cite[Theorem 5.4]{kiem-2004} and Proposition \ref{P2.6}, the intersection pairing of $\alpha,\beta\in \mathrm{IH}^\bullet (Y_{2r})$ is 
\beq\label{e40}\langle \alpha,\beta\rangle_{\mathrm{IH}^\bullet (Y_{2r})}= \int_{\widetilde{Y}_{2r}}\imath(\alpha)\cup\imath(\beta)|_{\widetilde{Y}_{2r}}=\int_{{Y}_{2r+1}}\imath(\alpha)\cup\imath(\beta)|_{{Y}_{2r+1}}\cdot y_{2r+1}\eeq
which is computed by \cref{e14} or \cref{e21}.

\begin{example}
Let $r=2$. By \cref{e26}, it is easy to see that $I\!H^2(Y_4)=\QQ (y_1+y_2+y_3+y_4)$. By \cref{e21}, we find that 
\beq\label{e39}\langle 1,y_1+y_2+y_3+y_4\rangle_{\mathrm{IH}^\bullet (Y_4)} = \int_{\widetilde{Y}_4} y_1+y_2+y_3+y_4 = 1.\eeq 
\end{example}

\begin{remark}
    We expect that using \emph{intrinsic Donaldson--Thomas theory}
    \cite{intrinsic-dt-i,intrinsic-dt-ii,intrinsic-dt-iii},
    one can generalize the work of \textcite{joyce-wall-crossing}
    to construct generalized intersection pairings for
    quasi-smooth Artin stacks.
    In particular, we should have an intrinsic class
    $a_{2r} \in \mathrm{H}_{4r-6} (\mathcal{Y}_{2r})$,
    using a natural choice of
    \emph{stability measure} \cite{intrinsic-dt-ii} on~$\mathcal{Y}_{2r}$,
    which is a generalized notion of stability for Artin stacks.
    The class~$a_{2r}$ should agree with the intersection pairing
    computed by both sides of \cref{e22}.
\end{remark}

\section{Moduli spaces of vector bundles with extra structure}

\label{sec-moduli}

Let~$C$ be a connected smooth projective curve over~$\mathbb{C}$.
We consider the moduli stack
\[
    \mathcal{M} = \coprod_{(r, d)} \mathcal{M}_{r, d}
\]
of vector bundles on~$C$, where
$(r, d) \in \{ (0, 0) \} \cup (\mathbb{Z}_{> 0} \times \mathbb{Z})$,
and each $\mathcal{M}_{r, d}$ is the connected component
of vector bundles of rank~$r$ and degree~$d$.
We also consider the $\mathbb{C}^\times$-rigidification
\[
    \mathcal{M}^{\mathrm{rig}} =
    \coprod_{(r, d)} \mathcal{M}^{\smash{\mathrm{rig}}}_{r, d} \ ,
\]
obtained from~$\mathcal{M}$ by removing scalar automorphisms
from all stabilizer groups.
For each $(r, d) \neq (0, 0)$,
we have the open substack
\[
    \mathcal{M}^{\mathrm{ss}}_{r, d} \subset
    \mathcal{M}^{\smash{\mathrm{rig}}}_{r, d}
\]
of semistable vector bundles.
It is a smooth projective variety when $r, d$ are coprime,
but is genuinely stacky when $r, d$ are not coprime.

In the following, we introduce three variants of
the moduli stack~$\mathcal{M}$ of vector bundles,
which all admit smooth morphisms to~$\mathcal{M}$.
Namely, we consider moduli stacks of
\emph{parabolic bundles},
\emph{pairs}, and
\emph{parabolic triples}.
These can all be used to define generalized intersection pairings
on~$\mathcal{M}^\mathrm{ss}_{r, d}$,
and we will compare the results later on.

\subsection{Parabolic bundles}

Let~$C$ be a curve as above,
and let~$p \in C$ be a $\mathbb{C}$-point.
We consider vector bundles on~$C$
with a parabolic structure at~$p$;
see \cite{mehta-seshadri-1980-moduli,seshadri-1977}.

\begin{definition}
    Let $\ell > 0$ be a fixed integer.
    A \emph{parabolic bundle} on~$C$ is a pair $(E, F)$, where
    \footnote{
        In the literature, such as in \cite{mehta-seshadri-1980-moduli},
        this is sometimes called a \emph{quasi-parabolic bundle},
        in which case a parabolic bundle is this data plus a choice
        of parabolic weights, as in \cref{def-par-weights}.
    }
    \begin{itemize}
        \item
            $E \to C$ is a vector bundle, and
        \item
            $F = (0 = F_0 \subset F_1 \subset \cdots \subset F_\ell = E |_p)$
            is a flag of subspaces of the fibre~$E |_p$.
    \end{itemize}
    Given a parabolic bundle $(E, F)$, we write
    \begin{itemize}
        \item
            $r = \operatorname{rank} E$,
        \item
            $d = \deg E$, and
        \item
            $f = (f_1, \dotsc, f_\ell)$,
            where $f_i = \dim F_i - \dim F_{i-1}$,
            and $|f| = \sum f_i = r$.
    \end{itemize}
    We call $(r, d, f)$ the \emph{type} of~$(E, F)$,
    and call~$f$ the (\emph{parabolic}) \emph{type} of~$F$.

    There is a moduli stack of parabolic bundles,
    \[
        \mathcal{P} =
        \coprod_{(r, d, f)} \mathcal{P}_{r, d, f} \ ,
        \qquad
        \mathcal{P}_{r, d, f} =
        \mathcal{M}_{r, d}
        \underset{[* / \mathrm{GL}_r]}{\times}
        [* / P_f] \ ,
    \]
    where $P_f \subset \mathrm{GL}_r$
    is the parabolic subgroup which is the stabilizer
    of a flag of type~$f$ in~$\mathbb{C}^r$,
    and the morphism $\mathcal{M}_{r, d} \to [* / \mathrm{GL}_r]$
    is the one classifying the vector bundle
    $\mathcal{U} |_{\mathcal{M}_{r, d} \times \{ p \}}$ of rank~$r$,
    where $\mathcal{U} \to \mathcal{M} \times C$ is the universal bundle.
    Note that although the definition of~$\mathcal{P}$
    depends on~$\ell$, we omit it from the notation.

    We also consider the $\mathbb{C}^\times$-rigidification
    $\mathcal{P}^{\mathrm{rig}} =
    \coprod_{(r, d, f)}
    \mathcal{P}^{\smash{\mathrm{rig}}}_{r, d, f}$
    of~$\mathcal{P}$.

    When $r = \ell$ and $f = 1^r = (1, \dotsc, 1)$,
    we omit~$f$ from the notation, and write
    \[
        \mathcal{P}_{r, d} = \mathcal{P}_{r, d, 1^r} \ ,
        \qquad
        \mathcal{P}^{\smash{\mathrm{rig}}}_{r, d}
        = \mathcal{P}^{\smash{\mathrm{rig}}}_{r, d, 1^r} \ .
    \]
\end{definition}

There is a family of interesting stability conditions
for parabolic bundles, parametrized by \emph{parabolic weights}.

\begin{definition}
    \label{def-par-weights}
    A \emph{parabolic weight} is a sequence of numbers
    \[
        c = \Bigl( \ 
            0 \leq c_1 \leq \cdots \leq c_\ell \leq 1
        \ \Bigr) \ .
    \]
    For a parabolic weight~$c$
    and a type $(r, d, f)$ with $r > 0$,
    define the slope function
    \[
        \mu_c (r, d, f) =
        \frac{1}{r} \biggl(
            d
            - \sum_{i=1}^{\ell} c_i f_i
        \biggr) \ .
    \]
    We have the $\mu_c$-semistable locus, which is an open substack
    \[
        \mathcal{P}^\mathrm{ss}_{r, d, f} (c) =
        \mathcal{P}^\mathrm{ss}_{r, d, f} (\mu_c) \subset
        \mathcal{P}^{\smash{\mathrm{rig}}}_{r, d, f} \ .
    \]
    We introduce the following terminology:
    \begin{itemize}
        \item
            We say that~$c$ is \emph{generic} for the type $(r, d, f)$,
            if for any type $(r', d', f')$ with $0 < r' < r$,
            if we have $\mu_c (r', d', f') \neq \mu_c (r, d, f)$.
            In this case,
            there are no strictly $\mu_c$-semistable parabolic bundles in this type,
            and $\mathcal{P}^\mathrm{ss}_{r, d, f} (c)$ is a smooth projective variety.

        \item
            We say that~$c$ is \emph{close to~$0$} for the type $(r, d, f)$,
            if for any type $(r', d', f')$ with $0 < r' < r$
            and $d'/r' < \mathrel{(>)} d/r$,
            we have $\mu_c (r', d', f') < \mathrel{(>)} \mu_c (r, d, f)$.
            In this case, we have an inclusion
            $\mathcal{P}^\mathrm{ss}_{r, d, f} (c) \subset
            \mathcal{P}^\mathrm{ss}_{r, d, f} (0)$
            as an open substack.
    \end{itemize}
    Again, when $r = \ell$ and $f = 1^r$, we omit~$f$ from the notation.
    In this case, we also say that~$c$ is \emph{generic} or \emph{close to~$0$}
    for the type $(r, d)$ if it is so for the type $(r, d, 1^r)$.
\end{definition}

\subsection{Pairs and parabolic triples}

\begin{definition}
    Fix an integer $N$.
    A \emph{pair} on~$C$ is the data $(E, V)$, where
    \begin{itemize}
        \item
            $E \to C$ is a vector bundle, and
        \item
            $V \subset \mathrm{H}^0 (C, E (N))$
            is a linear subspace.
    \end{itemize}
    We write $v = \dim V$, and call
    $(r, d, v)$ the \emph{type} of the pair $(E, V)$.
    There is a moduli stack
    \[
        \mathcal{M}' =
        \coprod_{(r, d, v)} \mathcal{M}'_{r, d, v}
    \]
    of pairs on~$C$, obtained as an open substack of the stack in
    \cite[Definition~8.4]{joyce-wall-crossing}
    consisting of points $(E, V, \rho)$ using the notation there, such that
    $\rho \colon V \to \mathrm{H}^0 (C, E (N))$ is injective.

    We also consider its
    $\mathbb{C}^\times$-rigidification,
    $\mathcal{M}'^{\mathrm{rig}} =
    \coprod_{(r, d, v)} \mathcal{M}'^{\smash{\mathrm{rig}}}_{r, d, v}$.
\end{definition}

Note that a difference between our definition of pairs
and Joyce's \cite{joyce-song-2012,joyce-wall-crossing}
is that in Joyce's formalism, a pair consists of~$E$ and a map
$\rho \colon V \to \mathrm{H}^0 (C, E (N))$,
instead of a subspace.
Our definition is equivalent to requiring that~$\rho$ is injective,
and the purpose is to make the forgetful morphism
$\mathcal{M}'_{r, d, 1} \to \mathcal{M}_{r, d}$
a projective bundle,
so that it is more convenient for defining the map~$\Pi$
in \cref{sec-map-pi} below.

\begin{definition}
    \label{def-triples}
    A \emph{parabolic triple} is a triple $(E, F, V)$, where
    \begin{itemize}
        \item
            $(E, F)$ is a parabolic bundle on~$C$, and
        \item
            $(E, V)$ is a pair on~$C$.
    \end{itemize}
    Let $f$ be the parabolic type of~$F$
    and $v = \dim V$ as before,
    and call $(r, d, f, v)$ the \emph{type} of~$(E, F, V)$.

    There is a moduli stack
    of parabolic triples on~$C$,
    \[
        \mathcal{P}' =
        \mathcal{P} \underset{\mathcal{M}}{\times} \mathcal{M}' =
        \coprod_{(r, d, f, v)} \mathcal{P}'_{r, d, f, v} \ ,
        \qquad
        \mathcal{P}'_{r, d, f, v} =
        \mathcal{P}_{r, d, f}
        \underset{\mathcal{M}_{r, d}}{\times} \mathcal{M}'_{r, d, v} \ .
    \]
    We also consider its $\mathbb{C}^\times$-rigidification
    $\mathcal{P}'^{\mathrm{rig}} =
    \coprod_{(r, d, f, v)} \mathcal{P}'^{\smash{\mathrm{rig}}}_{r, d, f, v}$.

    When $r = \ell$ and $f = 1^r = (1, \dotsc, 1)$,
    we omit~$f$ from the notation, and write
    \[
        \mathcal{P}'_{r, d, v}
        = \mathcal{P}'_{r, d, 1^r, v} \ ,
        \qquad
        \mathcal{P}'^{\smash{\mathrm{rig}}}_{r, d, v}
        = \mathcal{P}'^{\smash{\mathrm{rig}}}_{r, d, 1^r, v} \ .
    \]
\end{definition}

\begin{definition}
    \label{def-triple-slope}
    For a parabolic weight~$c$,
    a parameter $\varepsilon \in \mathbb{R}_{> 0}$,
    a type of pairs $(r, d, v)$,
    and a type of triples $(r, d, f, v)$,
    with $(r, d, v) \neq (0, 0, 0)$ in both cases,
    define the slope functions
    \begin{align*}
        \mu^\varepsilon (r, d, v)
        & =
        \frac{d + \varepsilon v}{r} \ ,
        \\
        \mu_c^\varepsilon (r, d, f, v)
        & =
        \frac{1}{r} \biggl(
            d
            - \sum_{i=1}^{\ell} c_i f_i
            + \varepsilon v
        \biggr) \ .
    \end{align*}
    We write
    $\mu^+ = \mu^\varepsilon$ and
    $\mu_c^+ = \mu_c^\varepsilon$
    for $\varepsilon > 0$ sufficiently small
    (depending on~$c$ in the latter case);
    we define the slope to be~$+\infty$ when $r = 0$.
    We have the $\mu^+$- and $\mu_c^+$-semistable loci
    \begin{alignat*}{2}
        \mathcal{M}'^{\mathrm{ss}}_{r, d, v}
        & =
        \mathcal{M}'^{\mathrm{ss}}_{r, d, v} (\mu^+)
        && \subset
        \mathcal{M}'^{\smash{\mathrm{rig}}}_{r, d, v} \ .
        \\
        \mathcal{P}'^{\mathrm{ss}}_{r, d, f, v} (c)
        & =
        \mathcal{P}'^{\mathrm{ss}}_{r, d, f, v} (\mu_c^+)
        && \subset
        \mathcal{P}'^{\smash{\mathrm{rig}}}_{r, d, f, v} \ .
    \end{alignat*}
    When $N > 2g - 2 - d/r$,
    the stacks $\mathcal{M}'^{\mathrm{ss}}_{r, d, v}$
    and $\mathcal{P}'^{\mathrm{ss}}_{r, d, f, v} (c)$
    are smooth.
    This is because for a semistable vector bundle~$E$ of slope $d/r$,
    we have $\mathrm{H}^1 (C, E (N)) = 0$,
    so that $\mathrm{Ext}^2 ((E, V), (E, V)) = 0$
    for pairs $(E, V)$ in a bigger abelian category of pairs of coherent sheaves;
    see \cite[Eq.~(8.30)]{joyce-wall-crossing}.

    In particular, if additionally $v = 1$,
    then all semistable pairs and triples are stable,
    and $\mathcal{M}'^{\mathrm{ss}}_{r, d, 1}$ and $\mathcal{P}'^{\mathrm{ss}}_{r, d, f, 1} (c)$
    are smooth projective varieties.
\end{definition}

\section{Generalized intersection pairings by parabolic bundles and by partial desingularizations}

\label{sec-par}

\def\bcG{\overline{\cG}}
\def\modext{M^{\mathrm{ext}}}
\def\tmu{\tilde{\mu}}
\def\tcM{\widetilde{\cM}}

In this section, we recall the generalized intersection pairings on the moduli stack $\cM_{r,d}^{\mathrm{ss}}$ of rank $r$ semistable bundles of degree $d$ over a smooth projective curve $C$ of genus $g\ge 2$ by the partial desingularization (cf. \cite{kirwan-1985-partial}) and by the the moduli space of stable parabolic bundles with parabolic weights close to zero (cf. \cite{jeffrey-kirwan-1998-curves}).
We further recall from \cite{jeffrey-kiem-kirwan-woolf-2006-intersection} a systematic method to compute the difference of the two generalized intersection pairings. 
See \S\ref{S8.2} for explicit computations when $r=2$.

\subsection{Generalized intersection pairings by stable parabolic bundles}\label{S3.2}

As a topological stack, $\cM^{\mathrm{ss}}_{r,d}$ is the symplectic reduction 
of the space $\cA_{r,d}$ of unitary connections on a fixed hermitian vector bundle of rank $r$ and degree $d$ by the projective linear gauge group $\bcG=\cG/\CC^{\times}$. 
Fix $p\in C$ and let $\bcG_0$ be the kernel of the restriction map 
$\bcG\lra \bcG|_p=\mathrm{PU}(r)=:{K}.$
The symplectic reduction of $\cA_{r,d}$ by $\bcG_0$ is the \emph{extended moduli space} in \cite{Jeff} which we denote by $\modext$. By the residual action of $K$, the coarse moduli space of $\cM^{\mathrm{ss}}_{r,d}$ is now the symplectic reduction
$M_{r,d}\cong \mu^{-1}(0)/K$
of $\modext$ by $K$ where $\mu:\modext\to \bk^*\cong \bk$ denotes the moment map. 

Let $\bcG^{\CC}$ and $\bcG^{\CC}_0$ denote the complexifications of $\bcG$ and $\bcG_0$ respectively. 
By \cite{AB, Kir84}, the moduli stack $\cM^{\mathrm{ss}}_{r,d}$ 
is an open substack of the holomorphic quotient
$\cA_{r,d}/\bcG^{\CC}$ 
and $\modext\subset \cA_{r,d}/\bcG^{\CC}_0$ is open in the moduli stack of pairs of a holomorphic vector bundle
$E$ and a framing $E|_p\cong \CC^r$ at $p$. 
For sufficiently small general $c$ in the positive Weyl chamber of $K$, we have an isomorphism as topological stacks 
$$\cP_{r,d}^{\mathrm{s}}(c)=\cP_{r,d}^{\mathrm{ss}}(c)\cong [\mu^{-1}(c)/T]$$
of the moduli stack of $c$-stable parabolic bundles on $C$ with the quotient stack of $\mu^{-1}(c)$ where $T$ denotes a maximal torus of $K$. See \cite{Jeff} for more details.

As $c$ is general and close to $0$, the forgetful morphism $$\omega:\cP_{r,d}^{\mathrm{s}}(c)\lra \cM^{\mathrm{ss}}_{r,d}$$
is smooth and $\cP_{r,d}^{\mathrm{s}}(c)$ is proper \DM. 
The Euler class of the relative tangent bundle of $\omega$ is the product $\cD_r$ of positive roots of $\mathrm{SL}_r$ and a fibre over $\cM_{r,d}^{\mathrm{s}}$ is the flag variety $\Fl(r)$ whose Euler characteristic is $r!$. Thus \cref{8} gives us a generalized intersection pairing  
\beq\label{17}
\mathrm{H}^\bullet (\cM^{\mathrm{ss}}_{r,d})\lra \QQ, \quad 
\xi\mapsto 
\langle \xi\rangle_{\mathrm{JK}_{r,d}(c)}:=\int_{\cP^{\mathrm{s}}_{r,d}(c)} \frac{e(\T_\omega)}{w(\omega)}\cup \xi|_{\cP^{\mathrm{s}}_{r,d}(c)}=\int_{\mu^{-1}(c)/ T} \frac{\cD_r}{r!}\cup \xi|_{\mu^{-1}(c)/ T}.
\eeq
The computation of \cref{17} was worked out by Jeffrey and Kirwan in \cite[Proposition 8.4]{jeffrey-kirwan-1998-curves} 
by computing the wall crossing terms as we vary $c$ and then using the periodicity.
For convenience, we call \cref{17} the \emph{Jeffrey--Kirwan pairing}.
Equivalently, we can think of the Jeffrey--Kirwan pairing \cref{17} 
as a homology class (cf. Definition \ref{def-jkrd}) 
\beq\label{e51}
\mathrm{JK}_{r,d}(c)=\Omega (P_{r,d}(c))\in \mathrm{H}_\bullet(\cM_{r,d}^{\mathrm{ss}})\eeq 
where $P_{r,d}(c)=[\cP^{\mathrm{s}}_{r,d}(c)]\in \mathrm{H}_\bullet(\cP^{\mathrm{s}}_{r,d}(c))$ denotes the fundamental class and 
\beq\label{e53} \Omega:=\omega_*\left((-)\cap \frac{e(\T_\omega)}{w(\omega)}\right)= \omega_*\left((-)\cap \frac{\cD_r}{r!}\right):\mathrm{H}_\bullet(\cP^{\mathrm{s}}_{r,d}(c))\lra \mathrm{H}_\bullet(\cM^{\mathrm{ss}}_{r,d}).\eeq

\subsection{Generalized intersection pairings on partial desingularizations}\label{S3.1}
By \cite{Newstead, mumford-fogarty-kirwan-1994}, $\cM^{\mathrm{ss}}_{r,d}$ is the quotient stack 
\beq\label{22} \cM^{\mathrm{ss}}_{r,d}=[Q_{r,d}^{\mathrm{ss}}/G]\eeq
of the semistable part of a smooth projective scheme $Q_{r,d}$ 
by a linearized action of a reductive group $G$.  
In \cite{kirwan-1985-partial}, Kirwan introduced an algorithm to resolve stacky points with infinite stabilizers in $[Q^{\mathrm{ss}}_{r,d}/G]$ by a sequence of smooth blowups. 
It proceeds as follows. We first find the locus of closed points in $Q^{\mathrm{ss}}_{r,d}$ whose stabilizer group has maximal dimension. It turns out to be a smooth $G$-invariant closed subvariety of $Q^{\mathrm{ss}}_{r,d}$. 
We blow up along this subvariety and then remove unstable points with respect to a linearization which is close to the pullback of that for $Q_{r,d}$. Kirwan proved that the maximal dimension of stabilizer groups in the semistable part of the blown-up space is strictly smaller than before. Therefore by blowing up and taking out the unstable points finitely many times, we end up with a smooth quasi-projective variety $\widetilde{Q}^{s}_{r,d}$ with only finite stabilizer groups and a morphism
\beq\label{e43}
\rho:{\widetilde{\cM}_{r,d}}^{\mathrm{ss}}=[\widetilde{Q}^{s}_{r,d}/G]\lra [Q_{r,d}^{\mathrm{ss}}/G]= \cM^{\mathrm{ss}}_{r,d}\eeq
where $\widetilde{\cM}_{r,d}^{\mathrm{ss}}$ is \DM with proper good moduli space $\widetilde{Q}_{r,d} \git G$. 
By pulling back a cohomology class by \cref{e43}, we thus have a generalized intersection pairing
\beq\label{e42}
\mathrm{H}^\bullet (\cM^{\mathrm{ss}}_{r,d})\lra \QQ \ ,\qquad \xi \longmapsto \langle \xi\rangle_{\mathrm{JKKW}_{r,d}}:=
\int_{\widetilde{\cM}_{r,d}^{\mathrm{ss}}} \xi|_{\widetilde{\cM}_{r,d}^{\mathrm{ss}}}
\eeq
which was computed in \cite{jeffrey-kiem-kirwan-woolf-2006-intersection}. For convenience, let us call it the \emph{JKKW pairing} and the associated homology class
\beq\label{e52}
\mathrm{JKKW}_{r,d}=\rho_*[{\widetilde{\cM}_{r,d}}^{\mathrm{ss}}]\in \mathrm{H}_\bullet (\cM^{\mathrm{ss}}_{r,d})\eeq 
will be called the \emph{JKKW class}.

Because we know \cref{17} by the residue formulas in \cite[Proposition 8.4]{jeffrey-kirwan-1998-curves}, 
the computation of \cref{e42} is attained by computing the differnce 
\beq\label{e44}
\int_{\widetilde{\cM}_{r,d}^{\mathrm{ss}}} \xi|_{\widetilde{\cM}_{r,d}^{\mathrm{ss}}} - \int_{\cP^{\mathrm{s}}_{r,d}(c)} \frac{\cD_r}{r!}\cup \xi|_{\cP^{\mathrm{s}}_{r,d}(c)}
\eeq
of \cref{e42} and \cref{17}. In the subsequent subsection, we recall the computation of  \cref{e44}  from \cite{jeffrey-kiem-kirwan-woolf-2006-intersection}.

\subsection{Comparison of the Jeffrey--Kirwan and JKKW pairings}

Let $\tmu:\widetilde{M}^{\mathrm{ext}}\to \bk^*$ denote the induced moment map on the blown-up space $\widetilde{M}^{\mathrm{ext}}$ obtained by the partial desingularization process applied to $\modext$.  Let $\widetilde{\cP}^{\mathrm{s}}_{r,d}(0)=\tmu^{-1}(0)/T$. 
By the smooth fibration
$$\omega:\widetilde{\cP}^{\mathrm{s}}_{r,d}(0)=\tmu^{-1}(0)/T\lra \tmu^{-1}(0)/K=\widetilde{\cM}^{\mathrm{ss}}_{r,d}$$
of \DM stacks (cf. \cite[(6.1)]{JKKW2}), we have
\beq\label{12} \int_{\tcM_{r,d}^{\mathrm{ss}}}\xi|_{\tcM_{r,d}^{\mathrm{ss}}}=
\int_{\tmu^{-1}(0)/ K}\xi|_{\tmu^{-1}(0)/ K}=\int_{\tmu^{-1}(0)/T} \frac{\cD_r}{r!}\cup \xi|_{\tmu^{-1}(0)/T}.\eeq
Equivalently, letting $\widetilde{P}_{r,d}(0)=[\widetilde{\cP}^{\mathrm{s}}_{r,d}(0)]\in \mathrm{H}_\bullet(\widetilde{\cP}^{\mathrm{s}}_{r,d}(0))$, we have
\beq\label{e54}
\Omega (\widetilde{P}_{r,d}(0))=[\widetilde{\cM}^{\mathrm{ss}}_{r,d}] \and \rho_*\Omega (\widetilde{P}_{r,d}(0))=\rho_*[\widetilde{\cM}^{\mathrm{ss}}_{r,d}]=\mathrm{JKKW}_{r,d} \eeq
where $\Omega$ denotes the pushforward by $\omega$ after capping with $\cD_r/r!$ as \cref{e53}.

Moreover, since the blowups modify the moment map only near 0 (cf. \cite[(6.2)]{JKKW2}), we have
\beq\label{13}
\int_{\tmu^{-1}(c)/T} \frac{\cD_r}{r!}\cup \xi|_{\tmu^{-1}(c)/T}
=\int_{\mu^{-1}(c)/T} \frac{\cD_r}{r!}\cup \xi|_{\mu^{-1}(c)/T}
\eeq
for a general $c$ in the positive Weyl chamber of $K$. 
Equivalently, we have $$\widetilde{P}_{r,d}(c)=P_{r,d}(c)\and \Omega(\widetilde{P}_{r,d}(c))=\Omega(P_{r,d}(c))=\mathrm{JK}_{r,d}(c).$$

Combining \cref{12} and \cref{13}, we find that the difference \cref{e44} of \cref{e42} and \cref{17} equals 
\beq\label{15}
\int_{\tmu^{-1}(0)/T} \frac{\cD_r}{r!}\cup \xi|_{\tmu^{-1}(0)/T}-\int_{\tmu^{-1}(c)/T} \frac{\cD_r}{r!}\cup \xi|_{\tmu^{-1}(c)/T}\eeq
which amounts to computing the wall crossing 
$\widetilde{P}_{r,d}(0) -  \widetilde{P}_{r,d}(c)$
for $\xi\cap \cD_r/r!$ with $\xi\in \mathrm{H}_\bullet(\cM^{\mathrm{ss}}_{r,d})$.

To compute \cref{15},
we need to find out the new walls between $0$ and $c$ that arise from the blowups and to compute the wall crossing term for each of the new wall. 
In \cite[\S3]{jeffrey-kiem-kirwan-woolf-2006-intersection}, it was shown that the new walls to be crossed are parameterized by certain directed graphs and the wall crossing terms were computed in \cite[\S7]{jeffrey-kiem-kirwan-woolf-2006-intersection}. 

See \cite{jeffrey-kiem-kirwan-woolf-2006-intersection} for more details and \S\ref{S8.2} for explicit computations in the rank 2 case. The above comparison may be summarized by the following diagram
\begin{equation*}
    \begin{tikzcd}[column sep={5em, between origins}, row sep={4em, between origins}]
        \widetilde{P}_{r, d} (0)
        \ar[rr, leftrightarrow, "\text{wall-crossing}"]
        \ar[d, mapsto, "\Omega"']
        && \widetilde{P}_{r, d} (c)
        \ar[d, mapsto, "="']
         \\
       {[\widetilde{\cM}^{\mathrm{ss}}_{r,d}]} 
        \ar[d, mapsto, "\rho_*"']
        && P_{r, d} (c)
        \ar[d, mapsto, "\Omega"']
        \\
        \mathrm{JKKW}_{r, d} 
        \ar[rr, dotted, dash]
        && \mathrm{JK}_{r, d}(c) \rlap{.}
    \end{tikzcd}
\end{equation*}

\section{Generalized intersection pairings by stable pairs and the Joyce class}

\label{sec-joyce}

This section provides background material on
Joyce's work \cite{joyce-wall-crossing}
defining enumerative invariants in abelian categories,
as homology classes of moduli stacks,
which we call the \emph{Joyce class}.
These classes satisfy wall-crossing formulae
expressed using the Joyce vertex algebra.

We also explain how our setting of \cref{sec-moduli},
especially the setting of parabolic bundles and triples,
fits into Joyce's framework,
so that Joyce's wall-crossing formulae
can be applied to parabolic bundles and triples.
These will be useful for comparing
the Joyce class for vector bundles on curves
with Jeffrey--Kirwan pairings later on.

\subsection{The Joyce class}

\subsubsection{Joyce vertex algebras}
\label{subsubsec-jva}

Joyce vertex algebras
(\cite{joyce-hall,joyce-wall-crossing,gross-joyce-tanaka-2022};
see also \cite{bu-vertex})
are vertex algebra structures defined on
the homology of moduli stacks of objects in linear categories.
These vertex algebra structures are used to
express wall-crossing relations
for enumerative invariants living in the homology of moduli stacks,
which is the main motivation for constructing such vertex algebras.

Given a $\mathbb{C}$-linear category~$\mathcal{A}$,
and a quasi-smooth moduli stack~$\mathcal{X}$ of objects in~$\mathcal{A}$,
under certain conditions, the homology
\[
    V = \mathrm{H}_{\bullet + 2 \operatorname{vdim}} (\mathcal{X}; \mathbb{Q})
\]
admits a vertex algebra structure,
called the \emph{Joyce vertex algebra},
where $\operatorname{vdim}$ denotes the virtual dimension of~$\mathcal{X}$,
seen as a locally constant function on~$\mathcal{X}$.
The vertex product is roughly given by
\[
    a_1 (z_1) \cdot a_2 (z_2) =
    \oplus_* \circ \mathrm{e}^{z_1 D_1 + z_2 D_2}
    \Bigl(
        (a_1 \boxtimes a_2) \cap e_z (\mathbb{T}_\oplus)
    \Bigr) \ ,
\]
where $a_1, a_2 \in V$ are homology classes,
$z_1, z_2$ are formal variables,
$\oplus \colon \mathcal{X}^2 \to \mathcal{X}$
is the direct sum map,
$\mathbb{T}_\oplus$ is its relative tangent complex,
and $e_z (\mathbb{T}_\oplus)$
is the $(\mathbb{C}^\times)^2$-equivariant Euler class of $\mathbb{T}_\oplus$,
with $z = (z_1, z_2)$ the equivariant parameters.
The operators $D_1, D_2 \colon V^{\otimes 2} \to V^{\otimes 2}$
are given by $D \otimes 1$ and $1 \otimes D$, respectively,
where $D \colon V \to V$ is the \emph{translation operator},
defined from the action
$\odot \colon [* / \mathbb{C}^\times] \times \mathcal{X} \to \mathcal{X}$
by taking the generator in
$\mathrm{H}_2 ([* / \mathbb{C}^\times]; \mathbb{Q})$
dual to the universal first Chern class in
$\mathrm{H}^2 ([* / \mathbb{C}^\times]; \mathbb{Q})$.
See \cite[Theorem~2.3.3]{bu-vertex} for a precise statement.

It is a general fact that given a vertex algebra~$V$,
the quotient
\[
    \bar{V} = V / D (V)
\]
admits a Lie algebra structure,
with Lie bracket
\[
    [a_1, a_2] = \mathrm{res}_{z_1 = z_2}
    \bigl( a_1 (z_1) \cdot a_2 (z_2) \bigr) \ .
\]
When $V = \mathrm{H}_{\bullet + 2 \operatorname{vdim}} (\mathcal{X}; \mathbb{Q})$
as above,
we often have a natural identification
\[
    \bar{V} \simeq
    \mathrm{H}_{\bullet + 2 \operatorname{vdim}} (\mathcal{X}^{\mathrm{rig}}; \mathbb{Q}) \ ,
\]
where $\mathcal{X}^{\mathrm{rig}}$ is the
$\mathbb{C}^\times$-rigidification of~$\mathcal{X}$.
See \cite[Theorem~4.8]{joyce-wall-crossing}.
Note that the virtual dimension of~$\mathcal{X}^{\mathrm{rig}}$
differs from that of~$\mathcal{X}$ by~$1$,
but as the variables~$z_i$ have degree~$-2$,
the residue operator also has degree~$-2$,
which cancels with this grading shift.
\footnote{
    To make the grading consistent,
    we use the unusual convention that
    the virtual dimension of the point
    $\{ 0 \} \subset \mathcal{X}^{\mathrm{rig}}$ is~$1$.
    But this is not important,
    as we never use this part of the Lie algebra in this paper.
}

\subsubsection{The Joyce class}
\label{subsubsec-joyce-inv}

Under the setting above,
under certain assumptions,
Joyce \cite[Theorem~5.7]{joyce-wall-crossing}
defines homological enumerative invariants,
which we call the \emph{Joyce class} and denote by
\[
    J_\alpha (\tau) \in
    \mathrm{H}_{2 \operatorname{vdim}} (\mathcal{X}^{\mathrm{ss}}_\alpha (\tau); \mathbb{Q}) \ ,
\]
where $\alpha \in \uppi_0 (\mathcal{X}) \setminus \{ 0 \}$
is a class corresponding to a connected component
$\mathcal{X}_\alpha \subset \mathcal{X}$,
$\tau$ is a stability condition,
and $\mathcal{X}^{\mathrm{ss}}_\alpha (\tau) \subset
\mathcal{X}^{\smash{\mathrm{rig}}}_\alpha$
is the $\tau$-semistable locus.

For example, in the setting of parabolic bundles,
as in \cref{sec-moduli},
$\alpha$ is a class $\alpha = (r, d, f)$,
and $\tau$ can be taken to be a slope function of the form
in \cref{def-par-weights}.
The settings of vector bundles, pairs, and parabolic triples, are similar.
We will verify in \cref{subsec-joyce-axioms}
that these all satisfy Joyce's assumptions.

The Joyce class satisfies the following properties:

\begin{itemize}
    \item
        If all semistable objects of class
        $\alpha$ are stable,
        so that the semistable locus
        $\mathcal{X}^{\mathrm{ss}}_\alpha \subset \mathcal{X}^{\mathrm{rig}}_\alpha$
        is a smooth (or quasi-smooth) proper algebraic space,
        we have
        \[
            J_\alpha (\tau) =
            [\mathcal{X}^{\mathrm{ss}}_\alpha (\tau)]^{(\mathrm{vir})} \ ,
        \]
        the (virtual) fundamental class of~$\mathcal{X}^{\mathrm{ss}}_\alpha (\tau)$.
    \item
        For two stability conditions~$\tau$ and~$\tau'$,
        the classes
        $J_\alpha (\tau)$ and $J_\alpha (\tau')$
        are related by the \emph{wall-crossing formula}
        \[
            J_\alpha (\tau') =
            \sum_{\alpha = \alpha_1 + \cdots + \alpha_n} {}
            \widetilde{U} (\alpha_1, \dotsc, \alpha_n; \tau, \tau') \cdot
            [ [ \dotsc [ J_{\alpha_1} (\tau),
            J_{\alpha_2} (\tau) ], \dotsc ] ,
            J_{\alpha_n} (\tau) ] \ ,
        \]
        where the sum is over all decompositions of~$\alpha$
        into non-zero classes~$\alpha_1, \dotsc, \alpha_n$,
        and $\widetilde{U} ({\cdots})$ are combinatorial coefficients
        defined in \cref{def-wcf-coeff} below.
        All classes in the formula are pushed forward from the semistable locus
        to $\mathcal{X}^{\mathrm{rig}}_\alpha$
        or $\mathcal{X}^{\mathrm{rig}}_{\alpha_i}$,
        and the Lie brackets $[-, -]$ are defined using
        the Joyce vertex algebra, as in \cref{subsubsec-jva}.
\end{itemize}
In slightly more detail,
the class $J_\alpha (\tau)$ is defined as follows.

Consider an auxiliary category~$\mathcal{A}'$
of \emph{pairs} in~$\mathcal{A}$,
with a moduli stack of objects~$\mathcal{X}'$.
For example, if~$\mathcal{A}$ is the category of vector bundles
(or parabolic bundles) on a curve,
as in \cref{sec-moduli},
then $\mathcal{A}'$ can be taken to be the category of
pairs (or parabolic triples).
For a stability condition~$\tau$ on~$\mathcal{A}$,
there is an induced stability condition $\tau^+$ on~$\mathcal{A}'$,
and for the class $(\alpha, 1)$ of pairs,
the virtual class
$J'_{\alpha, 1} (\tau^+) =
[\mathcal{X}'^{\mathrm{ss}}_{\alpha, 1} (\tau^+)]^\mathrm{vir}$
exists. The class $J_\alpha (\tau)$ is then given by
\begin{equation}
    \label{eq-def-joyce-inv}
    J_\alpha (\tau) =
    \pi_* \biggl(
        J'_{\alpha, 1} (\tau^+) \cap
        \frac{e (\mathbb{T}_\pi)}{w (\pi)}
    \biggr)
    +
    \hspace{-.5em}
    \sum_{\substack{
        \alpha = \alpha_1 + \cdots + \alpha_n \colon \\
        n > 1, \ \tau (\alpha_i) = \tau (\alpha)
    }}
    \hspace{-.5em}
    \frac{(-1)^n}{n!} \cdot
    \frac{r (\alpha_1)}{r (\alpha)} \cdot
    [ [ \dotsc [ J_{\alpha_1} (\tau),
    J_{\alpha_2} (\tau) ], \dotsc ] ,
    J_{\alpha_n} (\tau) ] \ ,
\end{equation}
where
$\pi \colon \mathcal{X}'^{\mathrm{ss}}_{\alpha, 1} (\tau^+)
\to \mathcal{X}^{\mathrm{ss}}_\alpha (\tau)$
is the forgetful morphism,
$\mathbb{T}_\pi$ is its relative tangent complex,
which is a vector bundle in this case,
and $e (\mathbb{T}_\pi)$ is its Euler class;
$w (\pi)$ is the Euler characteristic of the generic fibre of~$\pi$,
which is a projective space;
$r (-)$ is a function proportional to
the dimension of the vector space of pairs on a given object;
in the case of vector bundles and parabolic bundles on a curve,
$r (-)$ can be taken to be the rank of the vector bundle.
The Lie brackets are defined from the Joyce vertex algebra
for semistable objects of slope~$\tau (\alpha)$ in~$\mathcal{A}$.
The definition is by induction on~$r (\alpha)$.
See, again, \cite[Theorem~5.7]{joyce-wall-crossing} for details.

In particular,
$J_\alpha (\tau)$ defines a generalized intersection pairing on~$\mathcal{X}^{\mathrm{ss}}_\alpha (\tau)$
in the sense of \cref{51},
as $\mathcal{X}'^{\mathrm{ss}}_{\alpha, 1} (\tau^+)$
restricts to a projective bundle over the stable locus
$\mathcal{X}^{\mathrm{s}}_\alpha (\tau) \subset \mathcal{X}^{\mathrm{ss}}_\alpha (\tau)$,
and all Lie brackets in \cref{eq-def-joyce-inv}
lie in the image of $\oplus_*$,
so they restrict to zero in
$\mathrm{H}^{\smash{\mathrm{BM}}}_\bullet (\mathcal{X}^{\mathrm{s}}_\alpha (\tau))$,
because the image of $\oplus$ is disjoint from~$\mathcal{X}^{\mathrm{s}}_\alpha (\tau)$.

\begin{definition}
    \label{def-inv-notations}
    Under the setting of \cref{sec-moduli},
    let~$c$ be a parabolic weight for the type $(r, d)$ that is close to~$0$.
    We denote by
    \begin{alignat*}{2}
        J_{r, d}
        & \in \mathrm{H}_{2 \dim}
        (\mathcal{M}^{\smash{\mathrm{ss}}}_{r, d}; \mathbb{Q}) \ ,
        & \qquad
        P_{r, d, f} (c)
        & \in \mathrm{H}_{2 \dim}
        (\mathcal{P}^{\smash{\mathrm{ss}}}_{r, d, f} (0); \mathbb{Q}) \ ,
        \\
        J'_{r, d, v}
        & \in \mathrm{H}_{2 \operatorname{vdim}}
        (\mathcal{M}'^{\smash{\mathrm{ss}}}_{r, d, v}; \mathbb{Q}) \ ,
        & \qquad
        P'_{r, d, f, v} (c)
        & \in \mathrm{H}_{2 \operatorname{vdim}}
        (\mathcal{P}'^{\smash{\mathrm{ss}}}_{r, d, f, v} (0); \mathbb{Q})
    \end{alignat*}
    the Joyce classes for vector bundles,
    parabolic bundles,
    pairs, and parabolic triples, respectively,
    where we use the stability conditions given by the slope functions
    $\mu = d/r$, $\mu_c$, $\mu^+$, and~$\mu_c^+$,
    respectively, where the latter three are defined in
    \cref{def-par-weights,def-triple-slope}.
    In the case of $P_{r, d, f} (c)$ and $P'_{r, d, f, v} (c)$,
    for convenience, we push them forward along the inclusions
    $\mathcal{P}^{\smash{\mathrm{ss}}}_{r, d, f} (c)
    \hookrightarrow
    \mathcal{P}^{\smash{\mathrm{ss}}}_{r, d, f} (0)$
    and
    $\mathcal{P}'^{\smash{\mathrm{ss}}}_{r, d, f, v} (c)
    \hookrightarrow
    \mathcal{P}'^{\smash{\mathrm{ss}}}_{r, d, f, v} (0)$.

    As before, when $r = \ell$ and $f = 1^r$,
    we omit~$f$ from the notation, and write
    \[
        P_{r, d} (c) = P_{r, d, 1^r} (c) \ ,
        \qquad
        P'_{r, d, v} (c) = P'_{r, d, \smash{1^r}, v} (c) \ .
    \]
    For reference, we record that the (virtual) dimensions used above are given by the formula
    \begin{equation*}
        r^2 (g-1)
        \quad + \quad
        v (r (N+g-1) + d - v)
        \quad + \quad
        \sum_{1 \leq i < j \leq l} f_i f_j
        \quad + \quad
        1 \ ,
    \end{equation*}
    where we omit the second (resp.~third) term
    when the type does not involve~$v$ (resp.~$f$);
    the final term $+ 1$ comes from rigidification.
    The homological degrees of Joyce classes are
    twice the number given by this formula.
\end{definition}

Finally, we present the precise definition of
the combinatorial coefficients that appear
in the wall-crossing formula,
following \cite[\S4.1]{joyce-2008-configurations-iv}.

\begin{definition}
    \label{def-wcf-coeff}
    \allowdisplaybreaks
    Fix a class $\alpha$
    and a decomposition $\alpha = \alpha_1 + \cdots + \alpha_n$
    as above.

    Following Joyce,
    for two stability conditions~$\tau$ and~$\tau'$,
    define coefficients
    \begin{align*}
        S (\alpha_1, \dotsc, \alpha_n; \tau, \tau') & =
        \prod_{i=1}^{n-1} {} \left\{
            \mathrlap{ \begin{array}{ll}
                1, & \tau (\alpha_i) > \tau (\alpha_{i+1}) \text{ and } \\
                & \hspace{2em} \tau' (\alpha_1 + \cdots + \alpha_i)
                \leq \tau' (\alpha_{i+1} + \cdots + \alpha_n) \\
                -1, & \tau (\alpha_i) \leq \tau (\alpha_{i+1}) \text{ and } \\
                    & \hspace{2em} \tau' (\alpha_1 + \cdots + \alpha_i)
                    > \tau' (\alpha_{i+1} + \cdots + \alpha_n) \\
                0, & \text{otherwise}
            \end{array} }
            \hspace{21em}
        \right\} \ ,
        \\[1ex]
        U (\alpha_1, \dotsc, \alpha_n; \tau, \tau')
        & =
        \notag \\*[-.5ex]
        & \hspace{-7em}
        \sum_{ \leftsubstack[8em]{
            \\[-2ex]
            & 0 = a_0 < \cdots < a_m = n, \ 
            0 = b_0 < \cdots < b_\ell = m \colon
            \\[-1ex]
            & \text{Writing }
            \beta_i = \alpha_{a_{i-1}+1} + \cdots + \alpha_{a_i}
            \text{ for }
            i = 1, \dotsc, m, \\[-1ex]
            & \text{and }
            \gamma_i = \beta_{b_{i-1}+1} + \cdots + \beta_{b_i}
            \text{ for }
            i = 1, \dotsc, \ell, \\[-1ex]
            & \text{we have } \tau (\alpha_j) = \tau (\beta_i)
            \text{ for all } a_{i-1} < j \leq a_i , \\[-1ex]
            & \text{and } \tau' (\gamma_i) = \tau' (\alpha)
            \text{ for all } i = 1, \dotsc, \ell
        } } {}
        \frac{(-1)^{\ell-1}}{\ell} \cdot \biggl(
            \prod_{i=1}^{\ell}
            S (\beta_{b_{i-1}+1}, \dotsc, \beta_{b_i}; \tau, \tau')
        \biggr) \cdot
        \biggl(
            \prod_{i=1}^m \frac{1}{(a_i - a_{i-1})!}
        \biggr) \ ,
        \raisetag{4ex}
    \end{align*}
    and define coefficients
    $\widetilde{U} (\alpha_1, \dotsc, \alpha_n; \tau, \tau') \in \mathbb{Q}$,
    non-uniquely,
    as in \cite[Theorem~5.4]{joyce-2008-configurations-iv},
    by the property that
    \begin{multline*}
        \sum_{\sigma \in \mathfrak{S}_n}
        U (\alpha_{\sigma (1)}, \dotsc, \alpha_{\sigma (n)}; \tau, \tau') \cdot
        e_{\sigma (1)} * \cdots * e_{\sigma (n)}
        = \\*
        \sum_{\sigma \in \mathfrak{S}_n}
        \widetilde{U} (\alpha_{\sigma (1)}, \dotsc, \alpha_{\sigma (n)}; \tau, \tau') \cdot
        [[ \dotsc [ e_{\sigma (1)} , e_{\sigma (2)} ] , \dotsc ] , e_{\sigma (n)} ] \ ,
    \end{multline*}
    where the $e_i$ are formal symbols,
    $*$ is an associative product,
    and $[-, -]$ denotes the commutator for~$*$.
\end{definition}

\subsection{The Joyce class for vector bundles on curves}
\label{subsec-vb}

The Joyce class $J_{r, d}$ for vector bundles on curves
is computed explicitly by
the first author \cite{bu-2023-curves}.
When $r > 0$, it satisfies a \emph{regularized sum} formula
\begin{equation}
    J_{r, d} =
    \frac{1}{r} \cdot
    \sumbar_{\leftsubstack{
        & d = d_1 + \cdots + d_r \colon \\[-1ex]
        & \frac{d_1 + \cdots + d_i}{i} \leq \frac{d}{r}, \
        i = 1, \dotsc, r
    }}
    \frac{1}{\#\text{equal signs}} \cdot
    [ [ \dotsc [ J_{1, d_1},
    J_{1, d_2} ], \dotsc ] ,
    J_{1, d_r} ]
\end{equation}
in $\mathrm{H}_{2 \dim} (\mathcal{M}^{\smash{\mathrm{rig}}}_{r, d}; \mathbb{Q})$,
where we sum over all $(d_1, \dotsc, d_r) \in \mathbb{Z}^r$
satisfying the given conditions,
and `\#equal signs' refers to the number of equal signs
attained among the $r$~inequalities, which is at least~$1$ (when $i = r$).
The symbol `$\sumbar$' indicates that we are taking a regularized sum,
as defined in \cite{bu-2023-curves};
it is an infinite sum, usually divergent,
and is regularized by declaring, for example,
$1 + 2 + 3 + \cdots = -1/12$, etc.

From this formula, one can extract explicit pairings
of cohomology classes on
$\mathcal{M}^{\smash{\mathrm{ss}}}_{r, d}$
against the class $J_{r, d}$,
as iterated residues of meromorphic functions.
See \cite[Theorem~5.8]{bu-2023-curves} for the precise formula.
In particular, this formula computes intersection pairings
on~$\mathcal{M}^{\smash{\mathrm{ss}}}_{r, d}$
when $r, d$ are coprime.

\subsection{Verifying Joyce's axioms for triples}
\label{subsec-joyce-axioms}

We verify that Joyce's axioms,
\cite[Assumptions 5.1--5.3]{joyce-wall-crossing},
are satisfied for our setting of parabolic bundles and triples
introduced in \cref{sec-moduli}.
Note that as the category of parabolic bundles
embeds in the category of triples,
it is sufficient to verify these axioms for triples.

We first construct an abelian category~$\mathcal{A}$
that contains the category of triples.

\begin{definition}
    Let $C, p, \ell$ be as in \cref{sec-moduli}.
    Fix a slope $k \in \mathbb{Q}$
    and an integer $N > 2g - 2 - k$.

    Define an abelian category~$\mathcal{A}$ as follows.
    An object is a triple $(E, F, V)$, where
    \begin{itemize}
        \item
            $E$ is a semistable vector bundle on~$C$ of slope~$k$, and
        \item
            $F = (0 = F_0 \to F_1 \to \cdots \to F_\ell = E |_p)$
            is a sequence of linear maps of
            finite-dimensional $\mathbb{C}$-vector spaces.
        \item
            $V$ is a finite-dimensional $\mathbb{C}$-vector space,
            equipped with a map $V \to \mathrm{H}^0 (C, E (N))$.
    \end{itemize}
    A morphism $(E, F, V) \to (E', F', V')$
    is a morphism of sheaves
    $\varphi \colon E \to E'$, together with linear maps
    $F_i \to F'_i$ and $V \to V'$,
    making the relevant diagrams commute.

    Let $\mathcal{B} \subset \mathcal{A}$
    be the full subcategory of triples
    in the sense of \cref{def-triples},
    that is, $(E, F, V)$ such that
    all maps $F_i \to F_{i+1}$ and $V \to \mathrm{H}^0 (C, E (N))$ are injective.
\end{definition}

We use this choice of $\mathcal{A}, \mathcal{B}$ in
\cite[Assumption~5.1]{joyce-wall-crossing}.
We now go through the other non-obvious parts of
\cite[Assumption~5.1]{joyce-wall-crossing}:

\begin{itemize}
    \item
        For (b),
        we take $K (\mathcal{A}) = \mathbb{Z}^{\ell + 2}$,
        where we assign the type
        $(r, f, v)$ to an object $(E, F, V)$, where
        $r = \operatorname{rank} E$,
        $f = (f_1, \dotsc, f_\ell)$ with
        $f_i = \dim F_i - \dim F_{i-1}$,
        and $v = \dim V$.

    \item
        (c) asks for a moduli stack of objects of~$\mathcal{B}$,
        which we constructed in \cref{def-triples}.

    \item
        (d) asks for a derived structure on the moduli stack.
        As the stack is smooth,
        we use the trivial derived structure.

    \item
        (g) asks for a moduli stack of pairs in~$\mathcal{B}$.
        We can take $K = \{ k \}$ to be a singleton,
        $\mathcal{B}_k = \mathcal{B}$,
        $F_k (E, F, V) = \mathrm{H}^0 (C, E (N))$,
        and define the moduli stack to be a fibre product over~$\mathcal{M}$
        of the stack of triples defined above and the stack of pairs defined in
        \cite[Definition~8.4]{joyce-wall-crossing}.
\end{itemize}
It now remains to verify
\cite[Assumptions 5.2--5.3]{joyce-wall-crossing},
which asks for a family of
(weak) stability conditions with nice properties.

We use the slope functions $\mu_c^\varepsilon$
in \cref{def-par-weights}.
Verifying most of these properties are elementary;
the less obvious ones are:
\begin{itemize}
    \item
        Assumption 5.2 (b):
        Openness of the semistable locus.
    \item
        Assumption 5.2 (g--h):
        Properness of the semistable locus when
        $\text{stable} = \text{semistable}$,
        and the same property for certain auxiliary moduli stacks.
\end{itemize}
But all moduli problems considered here
can be seen as GIT quotients
with suitable stability conditions,
and these properties follow.

\section{Comparing generalized intersection pairings for vector bundles}\label{S:comparison}

\label{sec-comparison}

The main goal of this section is to compare the following
generalized intersection pairings on moduli spaces of vector bundles over a curve:
\begin{itemize}
    \item
        the Jeffrey--Kirwan class,
        obtained via parabolic bundles,
        as in \cref{sec-par}, and
    \item
        the Joyce class for vector bundles on curves,
        as in \cref{sec-joyce}.
\end{itemize}
The main tool for this comparison is Joyce's wall-crossing formula,
discussed in \cref{subsubsec-joyce-inv}.

\subsection{Projection homomorphisms}

We introduce two homomorphisms
between Joyce vertex algebras,
which will be main tools for the comparison later on.

\subsubsection{The map \texorpdfstring{$\Pi$}{Π}}
\label{sec-map-pi}

We adopt the notation from \cref{sec-moduli},
and assume that $N > 2g - 2 - d/r$,
so that the moduli stack of semistable parabolic triples
$\mathcal{P}'_{r, d, f, v}$ is smooth over
the semistable locus~$\mathcal{M}^\mathrm{ss}_{r, d}$.
Let
\[
    \pi \colon \mathcal{P}'_{r, d, f, v} |_{\mathcal{M}^\mathrm{ss}_{r, d}}
    \longrightarrow \mathcal{P}_{r, d, f} |_{\mathcal{M}^\mathrm{ss}_{r, d}}
\]
be the forgetful morphism,
which is a Grassmannian bundle
with fibre $\mathrm{Gr}_v (\mathbb{C}^{r (N - g + 1) + d})$.
When $v = 1$, it is a projective bundle with fibre
$\mathbb{P}^{r (N - g + 1) + d - 1}$.

Define the map
\begin{multline*}
    \Pi =
    \binom{r (N - g + 1) + d}{v}^{-1} \cdot
    \pi_* \bigl(
        (-) \cap e (\mathbb{T}_\pi)
    \bigr) \colon
    \\
    \mathrm{H}_{\bullet + 2 \dim} (\mathcal{P}'_{r, d, f, v} |_{\mathcal{M}^\mathrm{ss}_{r, d}}) \longrightarrow
    \mathrm{H}_{\bullet + 2 \dim} (\mathcal{P}_{r, d, f} |_{\mathcal{M}^\mathrm{ss}_{r, d}}) \ ,
\end{multline*}
where $e (-)$ denotes the Euler class.
Here, the binomial coefficient is the Euler characteristic
of the Grassmannian;
we will only be using the cases when $v = 0, 1$.

By
\textcite[\S 2.5]{gross-joyce-tanaka-2022},
stated more concisely in
\textcite[\S 2.3.5]{bu-vertex},
the map
\[
    \pi_* \bigl( (-) \cap e (\mathbb{T}_\pi) \bigr)
\]
is a vertex algebra homomorphism.
More precisely, for a fixed slope
$k \in \mathbb{Q}$,
we have a homomorphism between Joyce vertex algebras
\[
    \pi_* \bigl( (-) \cap e (\mathbb{T}_\pi) \bigr) \colon
    \bigoplus_{\substack{(r, d, f, v): \\ d = k r}}
    \mathrm{H}_{\bullet + 2 \dim} (\mathcal{P}'_{r, d, f, v} |_{\mathcal{M}^\mathrm{ss}_{r, d}})
    \longrightarrow
    \bigoplus_{\substack{(r, d, f): \\ d = k r}}
    \mathrm{H}_{\bullet + 2 \dim} (\mathcal{P}_{r, d, f} |_{\mathcal{M}^\mathrm{ss}_{r, d}}) \ .
\]
In particular,
$\Pi$ is also a homomorphism in this sense up to a factor,
and it descends to a map
$\mathrm{H}_{\bullet + 2 \dim} (\mathcal{P}'^{\smash{\mathrm{rig}}}_{r, d, f, v} |_{\mathcal{M}^\mathrm{ss}_{r, d}}) \to
\mathrm{H}_{\bullet + 2 \dim} (\mathcal{P}^{\smash{\mathrm{rig}}}_{r, d, f} |_{\mathcal{M}^\mathrm{ss}_{r, d}})$,
which is a homomorphism of Lie algebras
(for a fixed slope~$k$) up to a factor.

For a generic parabolic weight~$c$ close to~$0$, the induced morphism
$\pi \colon \mathcal{P}'^{\smash{\mathrm{ss}}}_{r, d} (c)
\to \mathcal{P}^{\smash{\mathrm{ss}}}_{r, d} (c)$
is a projective bundle over a smooth projective variety,
and we have the relation
\begin{equation}
    \label{eq-pi-triple-generic}
    \Pi (P'_{r, d, 1} (c)) =
    P_{r, d} (c)
\end{equation}
in $\mathrm{H}_{2 \dim} (\mathcal{P}^{\smash{\mathrm{ss}}}_{r, d} (c))$,
where we use the map
$\Pi \colon \mathrm{H}_{\bullet + 2 \dim} (\mathcal{P}'^{\smash{\mathrm{ss}}}_{r, d, 1} (c))
\to \mathrm{H}_{\bullet + 2 \dim} (\mathcal{P}^{\smash{\mathrm{ss}}}_{r, d} (c))$
defined analogously.

On the other hand, by the definition of the Joyce class,
writing $r_0 = r / \mathrm{gcd} (r, d)$, we have
\begin{align}
    \label{eq-pi-pair}
    \Pi (J'_{r, d, 1})
    & =
    \sum_{\substack{
        r = r_1 + \cdots + r_n: \\
        r_0 \, | \, r_i
    }} {}
    \frac{(-1)^{n-1}}{n!} \cdot \frac{r_1}{r} \cdot
    [ [ \ldots [
        J_{r_1, d_1}, J_{r_2, d_2}
        ], \dotsc ],
        J_{r_n, d_n}
    ] \ ,
    \\
    \label{eq-pi-triple-zero}
    \Pi (P'_{r, d, f, 1} (0))
    & =
    \sum_{\substack{
        f = f_1 + \cdots + f_n: \\
        r_0 \, | \, |f_i|
    }} {}
    \frac{(-1)^{n-1}}{n!} \cdot \frac{|f_1|}{r} \cdot
    [ [ \dotsc [ P_{f_1} (0),
    P_{f_2} (0) ], \dotsc ] ,
    P_{f_n} (0) ] \ ,
\end{align}
where we assume $r_i > 0$
and write $d_i = r_i \cdot d/r$,
and we use the shorthand
$P_{f_i} (c) = P_{r_i, d_i, f_i} (c)$ with
$r_i = |f_i|$ and $d_i = |f_i| \cdot d / r$.
Again, these can be interpreted as equalities in
$\mathrm{H}_{2 \dim} (\mathcal{M}^{\smash{\mathrm{ss}}}_{r, d})$
and
$\mathrm{H}_{2 \dim} (\mathcal{P}^{\smash{\mathrm{ss}}}_{r, d, f} (0))$,
respectively.

\subsubsection{The map \texorpdfstring{$\Omega$}{Ω}}

Let
\[
    \omega \colon \mathcal{P}'_{r, d, f, v} \to \mathcal{M}'_{r, d, v}
\]
be the forgetful morphism,
which is a smooth projective fibration whose fibre is
the space of flags of type~$f$ in $\mathbb{C}^r$.

Define the map
\begin{equation}
    \label{eq-def-omega}
    \Omega =
    \binom{r}{f_1, \dotsc, f_\ell}^{-1} \cdot
    \omega_* \bigl(
        (-) \cap
        e (\mathbb{T}_\omega)
    \bigr) \colon
    \mathrm{H}_{\bullet + 2 \dim} (\mathcal{P}'_{r, d, f, v})
    \longrightarrow
    \mathrm{H}_{\bullet + 2 \dim} (\mathcal{M}'_{r, d, v}) \ ,
\end{equation}
where the multinomial coefficient is the Euler characteristic
of the fibre of~$\omega$, which is the space of flags of type~$f$.
For \emph{full flags}, that is when $f_i \in \{ 0, 1 \}$ for all~$i$,
the coefficient becomes~$r!$.

We also have analogous maps
on open substacks, such as
\begin{equation*}
    \Omega \colon \mathrm{H}_{\bullet + 2 \dim} (\mathcal{P}'_{r, d, f, v} |_{\mathcal{M}^\mathrm{ss}_{r, d}})
    \longrightarrow \mathrm{H}_{\bullet + 2 \dim} (\mathcal{M}'_{r, d, v} |_{\mathcal{M}^\mathrm{ss}_{r, d}}) \ ,
\end{equation*}
defined by the same formula,
as well as on rigidifications.

Similarly to the case of~$\Pi$, by
\textcite[\S 2.5]{gross-joyce-tanaka-2022}
and \textcite[\S 2.3.5]{bu-vertex},
the map
\[
    \omega_* \bigl( (-) \cap e (\mathbb{T}_\omega) \bigr) \colon
    \mathrm{H}_{\bullet + 2 \dim} (\mathcal{P}')
    \longrightarrow
    \mathrm{H}_{\bullet + 2 \dim} (\mathcal{M}')
\]
is a homomorphism of Joyce vertex algebras,
and so is the version for a fixed slope,
as in the case of~$\Pi$ above.
In particular, in both versions,
$\Omega$ is a vertex algebra homomorphism up to a factor,
and on rigidifications,
a Lie algebra homomorphism up to a factor.

Note that by definition, we have
\begin{equation*}
    \Omega \circ \Pi = \Pi \circ \Omega \colon \quad
    \mathrm{H}_{\bullet + 2 \dim} (\mathcal{P}'_{r, d, f, v} |_{\mathcal{M}^\mathrm{ss}_{r, d}})
    \longrightarrow
    \mathrm{H}_{\bullet + 2 \dim} (\mathcal{M}_{r, d} |_{\mathcal{M}^\mathrm{ss}_{r, d}})
    \ ,
\end{equation*}
and similarly for the rigidified version.

When $N > 2g - 2 - d/r$,
taking $f = 1^r$ and $v = 1$,
the induced morphism
$\omega \colon \mathcal{P}'^{\smash{\mathrm{ss}}}_{r, d, 1} (0)
\to \mathcal{M}'^{\smash{\mathrm{ss}}}_{r, d, 1}$
is a flag variety bundle over a smooth projective variety,
and we have
\begin{equation}
    \label{eq-omega-triple-zero}
    \Omega (P'_{r, d, 1} (0)) =
    J'_{r, d, 1}
\end{equation}
in $\mathrm{H}_{2 \dim} (\mathcal{M}'^{\smash{\mathrm{ss}}}_{r, d, 1})$,
where we use the analogous map
$\Omega \colon \mathrm{H}_{\bullet + 2 \dim} (\mathcal{P}'^{\smash{\mathrm{ss}}}_{r, d, 1} (0))
\to \mathrm{H}_{\bullet + 2 \dim} (\mathcal{M}'^{\smash{\mathrm{ss}}}_{r, d, 1})$.

The generalized intersection pairings considered in \cref{S3.2}
can be rephrased as defining the following class:

\begin{definition}
    \label{def-jkrd}
    Let~$c$ be a generic parabolic weight for the type $(r, d)$ which is close to $0$.
    We define the \emph{Jeffrey--Kirwan class}
    \[
        \mathrm{JK}_{r, d} (c) =
        \Omega (P_{r, d} (c))
        \in \mathrm{H}_{2 \dim} (\mathcal{M}^{\smash{\mathrm{ss}}}_{r, d})
        \ .
    \]
    Here, we use the map
    $\Omega \colon \mathrm{H}_{\bullet + 2 \dim} (\mathcal{P}^{\smash{\mathrm{ss}}}_{r, d} (0))
    \to \mathrm{H}_{\bullet + 2 \dim} (\mathcal{M}^{\smash{\mathrm{ss}}}_{r, d})$
    defined by the same formula as \cref{eq-def-omega}.
\end{definition}

This definition can be compared to the following result,
which will be useful for computations later on.

\begin{theorem}
    \label{thm-omega-par-zero}
    We have
    \footnote{
        We note that this formula also appears in
        \textcite[Theorem~5.7]{moreira-parabolic};
        the class $P_{r, d} (0)$ is defined differently but equivalently there,
        and our definition directly using Joyce's framework
        enables us to use a more straightforward proof.
    }
    \[
        \Omega (P_{r, d} (0))
        = J_{r, d}
    \]
    in $\mathrm{H}_{2 \dim} (\mathcal{M}^{\smash{\mathrm{ss}}}_{r, d})$.
\end{theorem}

\begin{proof}
    The proof is essentially by
    applying $\Pi$ to both sides of
    \cref{eq-omega-triple-zero},
    then expanding using
    \cref{eq-pi-pair,eq-pi-triple-zero}.

    Namely, we use induction on~$r$.
    The case when $r = 1$ is clear,
    as parabolic structures are trivial.
    When $r > 1$,
    applying $\Omega$ to \cref{eq-pi-triple-zero},
    and using the induction hypothesis, we obtain
    \begin{align*}
        & \phantom{{} = {}}
        \Pi \circ \Omega (P'_{r, d, 1} (0))
        - \Omega (P_{r, d} (0))
        \\
        & =
        \sum_{\substack{
            1^r = f_1 + \cdots + f_n: \\
            n > 1, \ r_0 \, | \, |f_i|
        }} {}
        \frac{(-1)^{n-1}}{n!} \cdot
        \frac{r_1}{r} \cdot
        \frac{r_1! \cdots r_n!}{r!} \cdot
        [ [ \dotsc [ J_{r_1, d_1},
        J_{r_2, d_2} ], \dotsc ] ,
        J_{r_n, d_n} ]
        \\
        & =
        \sum_{\substack{
            r = r_1 + \cdots + r_n: \\
            n > 1, \ r_0 \, | \, r_i
        }} {}
        \frac{(-1)^{n-1}}{n!} \cdot
        \frac{r_1}{r} \cdot
        [ [ \dotsc [ J_{r_1, d_1},
        J_{r_2, d_2} ], \dotsc ] ,
        J_{r_n, d_n} ] \ ,
    \end{align*}
    where in the first step,
    we write $r_i = |f_i|$ and $d_i = |f_i| \cdot d / r$,
    and the extra coefficient comes from the fact that
    $\Omega$ is only a Lie algebra homomorphism up to a factor.
    In the second step,
    we collect all terms with the given $(r_1, \dotsc, r_n)$,
    and there are $r! / (r_1! \cdots r_n!)$ such terms.

    On the other hand, applying $\Pi$ to
    \cref{eq-omega-triple-zero},
    and using \cref{eq-pi-pair}, we see that
    $\Pi \circ \Omega (P'_{r, d} (0)) - J_{r, d}
    = \Pi (J'_{r, d, 1}) - J_{r, d}$
    is given by the same expression.
    This proves that $\Omega (P_{r, d} (0)) = J_{r, d}$.
\end{proof}

\subsection{Computing the wall-crossing}

We now aim to compute the difference
$\mathrm{JK}_{r, d} (c) - J_{r, d}$
of the Jeffrey--Kirwan class from \cref{def-jkrd}
and the Joyce class,
for a generic parabolic weight~$c$ close to $0$.

\subsubsection{Overview of computation}

\label{para-overview}
We aim to understand the diagram
\begin{equation*}
    \begin{tikzcd}[column sep={5em, between origins}, row sep={3em, between origins}]
        & P'_{r, d, 1} (c)
        \ar[rr, leftrightarrow, "\text{wall-crossing}"]
        \ar[dl, mapsto, "\Pi"']
        && P'_{r, d, 1} (0)
        \ar[dl, mapsto, "\smash{\Pi + (\text{corr.})}"']
        \ar[dd, mapsto, "\Omega"]
        \\
        P_{r, d} (c)
        \ar[rr, leftrightarrow, "\text{wall-crossing}"']
        \ar[dd, mapsto, "\Omega"']
        && P_{r, d} (0)
        \ar[dd, mapsto, "\Omega"']
        \\
        &&& J'_{r, d, 1}
        \ar[dl, mapsto, "\Pi + (\text{corr.})"]
        \\
        \mathrm{JK}_{r, d} (c)
        \ar[rr, dotted, dash]
        && J_{r, d}
        & \rlap{\hspace{2em},}
    \end{tikzcd}
\end{equation*}
so that we can compare the difference
$\mathrm{JK}_{r, d} (c) - J_{r, d}$.
Here, `corr.'\ refers to the correction terms
given by Lie brackets in \cref{eq-def-joyce-inv}.

We already know all solid edges except
those labelled with `wall-crossing',
and knowing any of the two is enough
for knowing the dotted edge.
Since both edges fit into the Joyce wall-crossing framework,
we describe both of them in the following sections.

Note that four of the classes in the diagram
are fundamental classes of smooth projective varieties:
\begin{equation*}
    \begin{tikzcd}[column sep=2em, row sep={3em, between origins}]
        & \mathcal{P}'^{\smash{\mathrm{ss}}}_{r, d, 1} (c)
        \ar[rr, dashed, leftrightarrow]
        \ar[dl, "\pi"']
        && \mathcal{P}'^{\smash{\mathrm{ss}}}_{r, d, 1} (0)
        \ar[dd, "\omega"]
        \\
        \mathcal{P}^{\smash{\mathrm{ss}}}_{r, d} (c)
        \\
        &&& \mathcal{M}'^{\smash{\mathrm{ss}}}_{r, d, 1} \rlap{ .}
    \end{tikzcd}
\end{equation*}
The map~$\pi$ here is a smooth projective bundle,
$\omega$ is a flag variety bundle,
and the dashed arrow is a change of stability conditions in GIT.

It is sometimes useful to interpret the space
$\mathcal{P}'^{\smash{\mathrm{ss}}}_{r, d, 1} (0)$
alternatively as follows.
Consider the slope function in \cref{def-par-weights},
\[
    \mu_c^\varepsilon (r, d, f, v) =
    \frac{1}{r} \biggl(
        d
        - \sum_{i=1}^{\ell} c_i f_i
        + \varepsilon v
    \biggr) \ ,
\]
but instead of choosing~$\varepsilon$ to be sufficiently small
(as is the case for~$\mu_c$),
we take a fixed value, such as $\varepsilon = 1/2$,
and consider the family of weights
$\varepsilon' c = (\varepsilon' c_1, \dotsc, \varepsilon' c_\ell)$, where~$\varepsilon' \in [0, 1]$.
Then for any weight~$c$, we have
\[
    \mathcal{P}'^{\smash{\mathrm{ss}}}_{r, d, 1} (0) =
    \mathcal{P}'^{\smash{\mathrm{ss}}}_{r, d, 1} (\mu_{\smash{\varepsilon' c}}^{1/2})
\]
when~$\varepsilon' > 0$ is sufficiently small.
This is because for both stability conditions
$\mu_0 = \mu_0^\varepsilon$ and~$\mu_{\smash{\varepsilon' c}}^{1/2}$,
a triple $(E, F, V)$ (where~$F$ is a full flag)
is stable if and only if the pair
$(E, V)$ is stable,
and both stability conditions satisfy
$\text{stable} = \text{semistable}$.

Roughly speaking, the path as $(\varepsilon, \varepsilon')$
moves along $(0^+, 0)$---$(1/2, 0)$---$(1/2, a)$
does not meet any walls if~$a$ is small enough.
Later on, we will do wall-crossing from $(0^+, 1)$
to this chamber.

\subsubsection{General wall-crossing}

We compute the wall-crossing
\[
    P_{r, d} (c) - P_{r, d} (0) \ ,
\]
which is the middle horizontal arrow of \cref{para-overview},
and use it to deduce
$\mathrm{JK}_{r, d} (c) - J_{r, d}$.

We now fix $(r, d)$, and take $\ell = r$.
Use the following notations:
\begin{itemize}
    \item
        Let $r_0 = r / \mathrm{gcd} (r, d)$.
    \item
        For a parabolic type $f \in \{ 0, 1 \}^r$,
        write $|f| = \sum_{i=1}^r f_i$.
    \item
        When $r_0 \mid |f|$,
        we use the shorthand $P_f (c) = P_{r', d', f} (c)$, where
        $r' = |f|$ and $d' = |f| \cdot d / r$.
\end{itemize}
We have Joyce's wall-crossing formula
\begin{align*}
    P_{r, d} (c) - P_{r, d} (0)
    & =
    \sum_{\substack{
        1^r = f_1 + \cdots + f_n: \\
        n > 1, \ r_0 \, | \, |f_i|
    }} {}
    \widetilde{U} (f_1, \dotsc, f_n; \mu_0, \mu_c) \cdot
    [ [ \dotsc [ P_{f_1} (0),
    P_{f_2} (0) ], \dotsc ] ,
    P_{f_n} (0) ] \ .
\end{align*}
Applying~$\Omega$, we obtain the following:

\begin{theorem}
    \label{thm-wcf-general-weights}
    Fix $(r, d)$, and let~$c$ be a generic parabolic weight.
    Then we have
    \begin{equation*}
        \mathrm{JK}_{r, d} (c) - J_{r, d} =
        \sum_{\substack{
            1^r = f_1 + \cdots + f_n: \\
            n > 1, \ r_0 \, | \, |f_i|
        }} {}
        \frac{r_1! \cdots r_n!}{r!} \cdot
        \widetilde{U} (f_1, \dotsc, f_n; \mu_0, \mu_c) \cdot
        [ [ \dotsc [ J_{r_1, d_1},
        J_{r_2, d_2} ], \dotsc ] ,
        J_{r_n, d_n} ] \ ,
    \end{equation*}
    where $|f_i| > 0$,
    $r_i = |f_i|$,
    and $d_i = r_i \cdot d / r$.
    \qed
\end{theorem}

Here, the extra coefficient
$r_1! \cdots r_n! / r!$
is due to the fact that
$\Omega$ is only a Lie algebra homomorphism up to a factor.

\subsubsection{Computing for specific weights}
\label{subsec-wcf-special-weights}

For special choices of~$c$,
the above wall-crossing can be made more explicit.
We compute this for the two generic weights
\[
    c^+ = \Bigl( \ c_i = 1 - \varepsilon^{i-1} \ \Bigr) \ ,
    \qquad
    c^- = \Bigl( \ c_i = \varepsilon^{r - i} \ \Bigr) \ ,
    \qquad
    i = 1, \dotsc, r \ ,
\]
with $\varepsilon > 0$ sufficiently small.
These weights are close to~$0$ in the sense of \cref{def-par-weights}.

\begin{theorem}
    \label{thm-wcf-special-weights}
    We have
    \begin{alignat*}{2}
        \mathrm{JK}_{r, d} (c^+) - J_{r, d}
        & =
        \sum_{\substack{
            r = r_1 + \cdots + r_n: \\
            n > 1, \ r_0 \, | \, r_i
        }} {}
        & \frac{(-1)^{n-1}}{n!} \cdot \frac{r_1}{r} \cdot
        [[ \dotsc [J_{r_1, d_1}, J_{r_2, d_2}], \dotsc, J_{r_n, d_n}]]
        & \ ,
        \\
        \mathrm{JK}_{r, d} (c^-) - J_{r, d}
        & =
        \sum_{\substack{
            r = r_1 + \cdots + r_n: \\
            n > 1, \ r_0 \, | \, r_i
        }} {}
        & \frac{r_1}{r} \cdot
        [[ \dotsc [J_{r_1, d_1}, J_{r_2, d_2}], \dotsc, J_{r_n, d_n}]]
        & \ ,
    \end{alignat*}
    where $r_0 = r / {\operatorname{gcd} (r, d)}$,
    $r_i > 0$,
    and $d_i = r_i \cdot d / r$.
\end{theorem}

\begin{proof}
    Fix a decomposition $1^r = f = f_1 + \cdots + f_n$,
    and write $r_i = |f_i|$.

    Following the definitions (by a standard but tedious combinatorics argument),
    we can show that a valid choice of the $\widetilde{U}$ coefficients is
    \begin{align*}
        \widetilde{U} (f_1, \dotsc, f_n; \mu_0, \mu_{c^+})
        & =
        \begin{cases}
            (-1)^{n-1}/n!,
            & (f_1)_1 = 1,
            \\
            0,
            & \text{otherwise},
        \end{cases}
        \\
        \widetilde{U} (f_1, \dotsc, f_n; \mu_0, \mu_{c^-})
        & =
        \begin{cases}
            1/n!,
            & (f_1)_r = 1,
            \\
            0,
            & \text{otherwise}.
        \end{cases}
    \end{align*}
    For~$c^+$, we then have
    \begin{align*}
        P_{r, d} (c^+) - P_{r, d} (0)
        & =
        \sum_{\substack{
            1^r = f_1 + \cdots + f_n: \\
            n > 1, \ (f_1)_1 = 1, \ r_0 \, | \, |f_i|
        }} {}
        \frac{(-1)^{n-1}}{n!} \cdot
        [ [ \dotsc [ P_{f_1} (0),
        P_{f_2} (0) ], \dotsc ] ,
        P_{f_n} (0) ] \ .
    \end{align*}
    Applying~$\Omega$, we obtain
    \begin{align*}
        \mathrm{JK}_{r, d} (c^+) - J_{r, d}
        & =
        \sum_{\substack{
            1^r = f_1 + \cdots + f_n: \\
            n > 1, \ (f_1)_1 = 1, \ r_0 \, | \, |f_i|
        }} {}
        \frac{(-1)^{n-1}}{n!} \cdot
        \frac{r_1! \cdots r_n!}{r!} \cdot
        [ [ \dotsc [ J_{r_1, d_1},
        J_{r_2, d_2} ], \dotsc ] ,
        J_{r_n, d_n} ]
        \\
        & =
        \sum_{\substack{r = r_1 + \cdots + r_n: \\ n > 1, \ r_0 \, | \, r_i}} {}
        \frac{(-1)^{n-1}}{n!} \cdot \frac{r_1}{r} \cdot
        [ [ \dotsc [ J_{r_1, d_1},
        J_{r_2, d_2} ], \dotsc ] ,
        J_{r_n, d_n} ]
    \end{align*}
    where we write $d_i = r_i \cdot d / r$,
    and in the second step,
    we collect all terms with the given sequence
    $(r_1, \dotsc, r_n)$, and there are
    $(r - 1)! / [(r_1 - 1)! \, r_2! \cdots r_n!]$ such terms.
    The case of $c^-$ is analogous.
\end{proof}

We notice the coincidence that
\[
    \mathrm{JK}_{r, d} (c^+) = \Pi (J'_{r, d, 1}) \ .
\]
This seems to be related to the fact that when considering $c^+$,
the one-dimensional subspace in the flag has the most importance,
and this is somehow analogous to Joyce's formalism,
where we consider rank one subsheaves.

\subsubsection{Computing via wall-crossing for triples}

We can also approach the wall-crossing
$\mathrm{JK}_{r, d} (c) - J_{r, d}$
via triples, using the wall-crossing
$P'_{r, d} (c) - P'_{r, d} (0)$
on the top row in the diagram of~\cref{para-overview}.

We now assume that~$c$ is generic
in the following stronger sense:
For any non-empty subsets $I, J \subset \{1, \dotsc, r\}$,
we have $\sum_{i \in I} c_i / |I| \neq \sum_{j \in J} c_j / |J|$.

Write $\beta = (r, d, 1^r, 1)$.
In this case, it is convenient to use the slope function
$\mu'_c = \mu_{\smash{\epsilon' c}}^{1/2}$ in \cref{para-overview}
in place of~$\mu_0$.
For a decomposition $\beta = \beta_1 + \cdots + \beta_n$,
with $\beta_i = (r_i, d_i, f_i, v_i)$
and $d_i / r_i = d / r$,
we may take
\[
    \widetilde{U} (\beta_1, \dotsc, \beta_n; \mu'_c, \mu_c) =
    \begin{cases}
        1,
        & v_1 = 1 \text{ and }
        \mu_c (f_1) < \mu_c (f)
        < \mu_c (f_2) < \cdots < \mu_c (f_n) \ ,
        \\
        0,
        & \text{otherwise},
    \end{cases}
\]
where we write $\mu_c (f_i) = \mu_c (r_i, d_i, f_i, 0)$ for short,
so that
\begin{align*}
    P'_{r, d, 1} (c) - P'_{r, d, 1} (\mu'_c) =
    \hspace{-1em}
    \sum_{\substack{
        1^r = f_1 + \cdots + f_n: \\
        n > 1, \ r_0 \, | \, |f_i|, \\
        \mu_c (f_1) < \mu_c (f) < \mu_c (f_2) < \cdots < \mu_c (f_n)
    }} {}
    \hspace{-1em}
    [ [ \dotsc [ P'_{f_1, 1} (\mu'_c),
    P_{f_2} (c) ], \dotsc ] ,
    P_{f_n} (c) ] \ ,
\end{align*}
where we note that $P'_{r, d, 1} (\mu'_c) = P'_{r, d, 1} (0)$.
Applying $\Pi \circ \Omega$, we obtain the following:

\begin{theorem}\label{thm7.5} 
    We have
    \begin{align*}
        \mathrm{JK}_{r, d} (c) - \Pi (J'_{r, d, 1}) =
        \hspace{-2.5em}
        \sum_{\substack{
            1^r = f_1 + \cdots + f_n: \\
            n > 1, \ r_0 \, | \, |f_i|, \\
            \mu_c (f_1) < \mu_c (f) < \mu_c (f_2) < \cdots < \mu_c (f_n)
        }} {}
        \hspace{-2.5em}
        \frac{r_1! \cdots r_n!}{r!} \cdot
        \frac{r_1}{r} \cdot
        [ [ \dotsc [ \Pi (J'_{r_1, d_1, 1}),
        \mathrm{JK}_{r_2, d_2} (c) ], \dotsc ] ,
        \mathrm{JK}_{r_n, d_n} (c) ] \ ,
    \end{align*}
    where $|f_i| > 0$,
    $r_i = |f_i|$,
    and $d_i = r_i \cdot d / r$.
    \qed
\end{theorem}

From here, using \cref{eq-pi-pair} and by induction on rank,
it is possible to express the difference
$\mathrm{JK}_{r, d} (c) - J_{r, d}$
in terms of $J_{r', d'}$ for $r' < r$,
resulting in a formula in a similar form to \cref{thm-wcf-general-weights}.

\section{Computation in low ranks}

\label{sec-computation}

We demonstrate our comparison of different
generalized intersection pairings
for vector bundles on curves in low ranks,
where explicit computations of their difference are possible.

\subsection{Comparing the Jeffrey--Kirwan class with the Joyce class}

\subsubsection{Rank 2}\label{S8.1.1}

For $(r, d) = (2, 0)$,
the only term that can appear in the wall-crossing formula
in \cref{thm-wcf-general-weights} is
$[J_{1, 0}, J_{1, 0}] = 0$,
so that for any weight~$c$, we have
\[
    \mathrm{JK}_{2, 0} (c) = J_{2, 0} \ ,
\]
without correction terms.
Moreover, all wall-crossing and
correction terms in the diagram in \cref{para-overview}
map to zero in $\mathrm{H}_\bullet (\mathcal{M}^{\smash{\mathrm{ss}}}_{2, 0})$,
because they become $[J_{1, 0}, J_{1, 0}]$.

\subsubsection{Rank 3}\label{S8.1.2}

For $(r, d) = (3, 0)$,
\cref{thm-wcf-special-weights} gives
\[
    \mathrm{JK}_{3, 0} (c^\pm) - J_{3, 0} =
    \mp \frac{1}{6} \cdot [J_{2, 0}, J_{1, 0}] \ .
\]
In this case,
$c^\pm$ represent the only two chambers of generic weights,
separated by the wall given by $c_3 - c_2 = c_2 - c_1$.

The wall-crossing term $[J_{2, 0}, J_{1, 0}]$
has a fully explicit formula.
Namely, applying the formula in \cref{subsec-vb} gives
\[
    [J_{2, 0}, J_{1, 0}] =
    \frac{1}{2} \cdot
    \sumbar_{d=1}^{\infty} {}
    [[J_{1, -d}, J_{1, d}], J_{1, 0}] \ ,
\]
and the right-hand side can be computed using
\cite[Theorem~B.6]{bu-2023-curves}.
For example, an analogous computation to that in
\cite[Theorem~5.8]{bu-2023-curves} gives
\begin{align}
    & \biggl<
        [J_{2, 0}, J_{1, 0}] \ , \
        \Xi \biggl(
            \prod_{l=2}^\infty S_{1, 0, l}^{\smash{m_l}} \cdot
            \exp (\alpha S_{1, 2, 2})
        \biggr)
    \biggr> =
    \notag \\
    & \hspace{4em}
    \frac{1}{2} \cdot
    \mathrm{res}_{z_2} \circ \mathrm{res}_{z_1} \biggl\{
        \frac{-\alpha^{3g}}
        {( z_1 z_2 (z_2 - z_1) )^{2g-2} \cdot (1 - \exp (\alpha z_1))}
        \cdot {}
    \notag \\
    & \hspace{8em}
        \prod_{l=2}^\infty {} \biggl[
            \frac{1}{l!} \cdot \biggl(
                \Bigl(
                    \frac{-z_1 - z_2}{3}
                \Bigr)^l +
                \Bigl(
                    \frac{2z_1 - z_2}{3}
                \Bigr)^l +
                \Bigl(
                    \frac{-z_1 + 2z_2}{3}
                \Bigr)^l
            \biggr)
        \biggr]^{m_l}
    \biggr\} \ ,
    \label{eq-j20-j10}
\end{align}
where
$S_{1,0,l} \in \mathrm{H}^{2l} (\mathcal{M}_{3, 0})$
and
$S_{1,2,2} \in \mathrm{H}^{2} (\mathcal{M}_{3, 0})$
are certain tautological classes,
$\Xi \colon \mathrm{H}^{\bullet} (\mathcal{M}_{3, 0}) \to
\mathrm{H}^{\bullet} (\mathcal{M}^{\smash{\mathrm{rig}}}_{3, 0})$
is a retract of the canonical map in the other direction
defined in \cite[\S2.3]{bu-2023-curves},
$(m_2, m_3, \dotsc)$ is a sequence of non-negative integers,
of which all but finitely many are zero,
and $\alpha$ is a formal symbol,
treated as a complex variable in the residue formula.

From \cref{eq-j20-j10}, one can compute that
\begin{alignat*}{2}
    & \text{when } g = 2 \ ,
    & \qquad
    \bigl<
        [J_{2, 0}, J_{1, 0}] \ , \
        \Xi (S_{1, 0, 3} \cdot S_{1, 2, 2}^7)
    \bigr>
    & =
    \frac{70}{9} \ ;
    \\
    & \text{when } g = 3 \ ,
    & \qquad
    \bigl<
        [J_{2, 0}, J_{1, 0}] \ , \
        \Xi (S_{1, 0, 3}^3 \cdot S_{1, 2, 2}^{10})
    \bigr>
    & =
    -\frac{32200}{729} \ ;
\end{alignat*}
etc.
In particular, the wall-crossing term
$[J_{2, 0}, J_{1, 0}]$ is non-zero in general,
so the three classes
$\mathrm{JK}_{3, 0} (c^+)$, $\mathrm{JK}_{3, 0} (c^-)$,
and $J_{3, 0}$ are all different.

\subsubsection{Higher ranks}\label{S8.1.3}

In higher ranks,
\cref{thm-wcf-special-weights} gives
\begin{align*}
    \mathrm{JK}_{4, 0} (c^\pm) - J_{4, 0}
    & =
    \mp \frac{1}{4} \cdot [J_{3, 0}, J_{1, 0}]
    + \frac{1}{24} \cdot [[J_{2, 0}, J_{1, 0}], J_{1, 0}] \ ,
    \\
    \mathrm{JK}_{4, 2} (c^\pm) - J_{4, 2}
    & = 0 \ ,
    \\
    \mathrm{JK}_{5, 0} (c^\pm) - J_{5, 0}
    & =
    \mp \frac{3}{10} \cdot [J_{4, 0}, J_{1, 0}]
    \mp \frac{1}{10} \cdot [J_{3, 0}, J_{2, 0}]
    \\
    & \hspace{4em}
    + \frac{1}{15} \cdot [[J_{3, 0}, J_{1, 0}], J_{1, 0}]
    + \frac{1}{30} \cdot [[J_{2, 0}, J_{1, 0}], J_{2, 0}] \ ,
\end{align*}
etc. Similarly, all these terms have fully explicit formulae.

\subsection{Comparing JKKW pairings with Jeffrey--Kirwan pairings}
\label{S8.2}

In this subsection, we compare the Jeffrey--Kirwan pairings by parabolic bundles with the JKKW pairings by the partial desingularization when the rank is 2. 

Let $\tcM_{2,0}^{\mathrm{ss}}$ denote the partial desingularization of the rigidified moduli stack $\cM_{2,0}^{\mathrm{ss}}$ of rank 2 semistable bundles of even degree over a smooth projective curve $C$. We recall the computation of 
\beq\label{e31}\int_{\tcM^{\mathrm{ss}}_{2,0}}\xi|_{\tcM^{\mathrm{ss}}_{2,0}} \quad \text{for }\xi\in H^*(\cM^{\mathrm{ss}}_{2,0})\eeq
in \cite[\S8]{jeffrey-kiem-kirwan-woolf-2006-intersection} which contains a few typos, for reader's convenience.  

To be explicit, we compute the intersection pairings of  
\beq\label{e36}\xi=a_2^mf_2^n/n!\in H^*(\cM^{\mathrm{ss}}_{2,0}) \quad \text{with }2m+n=4g-3\eeq
where $a_2, f_2$ are the K\"unneth components of the second Chern class of the universal bundle of degree $4$ and $2$ respectively. 
Recall from \cite{jeffrey-1994-extended} that there is a moment map $\mu:M^{\mathrm{ext}}\to \bk^*$ for the extended moduli space $M^{\mathrm{ext}}$ with $K=SU(2)$ and $\bk=\mathrm{Lie}(K)$ such that 
$$\cP^{\mathrm{s}}_{2,0}(c)\cong \mu^{-1}(c)/T$$ for $c=0^+$.
Let $\widetilde{M}^{\mathrm{ext}}$ be the partial desingularization and $\tilde{\mu}$ be the moment map for the induced action so that the partial desingularization is 
$$\widetilde{\cM}^{\mathrm{ss}}_{2,0}\cong \tilde{\mu}^{-1}(0)/K.$$
By \cite[Proposition 8.4]{JK}, the Jeffery--Kirwan intersection pairing 
\beq\label{e28} \int_{\mathrm{JK}_{2,0}(c)} \xi
=\int_{{\mu}^{-1}(c)/T} \frac{2t}{2}\cup \xi|_{{\mu}^{-1}(c)/T}
=\int_{\tilde{\mu}^{-1}(c)/T} \frac{2t}{2}\cup \xi|_{\tilde{\mu}^{-1}(c)/T}\eeq 
is equal to\footnote{Note that there is a sign error in the first line of \cite[Theorem 24]{JKKW2}.} 
\beq\label{e30}\frac{(-1)^{g-m}}{2^{2m-g+1}}\frac{B_{2g-2-2m}(1)}{(2g-2-2m)!}\eeq
where the Bernoullli numbers are defined by
\[ \frac{te^{\epsilon t}}{e^t-1}=\sum_n B_n(\epsilon)\frac{t^n}{n!}.\]
Note that \cref{e30} vanishes when $m\ge g$.

Note that  \cref{e31} is the sum of \cref{e30} and 
\beq\label{15b} \int_{\tilde{\mu}^{-1}(0)/T} \frac{2t}{2}\cup \xi|_{\tilde{\mu}^{-1}(0)/T} - \int_{\tilde{\mu}^{-1}(c)/T} \frac{2t}{2}\cup \xi|_{\tilde{\mu}^{-1}(c)/T}.\eeq 
We compute \cref{15b} as the sum of two wall crossing terms arising from the two blowups in the partial desingularization by the method. 

\subsubsection{The first blowup}\label{S5.2.1}
By the recipe and notation in \cite[\S8.1]{JKKW2}, the wall crossing term from the first blowup is 
\beq\label{e32} \int_{\PP^{g-1}\times \Delta} \res_{Y=0}\frac{Y^{2m}\frac{(-2\gamma)^n}{n!} (-2Y^2)}{(-y+2Y)(y-2Y)^g(y-4Y)^g}\eeq
where $\Delta=(S^1)^{2g}$ since $a_2|_\Delta=Y^2$ and $f_2|_\Delta=-2\gamma=-2\sum_{i=1}^gd_id_{i+g}$.
By using
\[ \frac{1}{(y-2Y)^{g+1}}=(-2Y)^{-g-1}\left(1-\frac{y}{2Y}\right)^{-g-1}\]
\[=(-2Y)^{-g-1}\sum_{k\ge 0}\binom{-g-1}{k} \left( \frac{-y}{2Y}\right)^k
=(-2Y)^{-g-1}\binom{g+k}{k}\frac{y^k}{Y^k}\]
and 
\[ \frac{1}{(y-4Y)^{g}}=(-4Y)^{-g}\left(1-\frac{y}{4Y}\right)^{-g}\]
\[=(-4Y)^{-g}\sum_{l\ge 0}\binom{-g}{l} \left( \frac{-y}{4Y}\right)^l
=(-4Y)^{g}\binom{g+l-1}{l}\frac{y^l}{Y^l}, \]
we find that \cref{e32} is 
$$\delta_{g,n}(-2)^g\left(\int_\Delta\prod_{i=1}^gd_id_{i+g}\right)\sum_{k+l=g-1}\binom{g+k}{k}\binom{g+l-1}{l}2^{-k-2l}$$
\beq\label{e33} =\delta_{g,n}2^{-3g+1}\sum_{k+l=g-1}2^{-l}\binom{g+k}{k}\binom{g+l-1}{l}\eeq 
where $\delta_{g,n}$ is the Kronecker symbol  and $\int_{(S^1)^{2g}}\prod_{i=1}^gd_id_{i+g}=1.$

Note that the wall crossing term vanishes when $n\ne g$.
\begin{example}
When $g=m=n=3$, the wall crossing term from the first blowup is 
\beq\label{e34}
2^{-8}\sum_{k=0}^22^{k-2}\binom{k+3}{3}\binom{4-k}{2}=\frac{35}{2^9}\ne 0.\eeq 
\end{example}

\subsubsection{The second blowup}
By the recipe and notation in \cite[\S8.2]{JKKW2}, we find that the wall crossing term from the second blowup is 
\beq\label{e35}
\int_{\PP W_+} \res_{Y=0} \frac{Y^{2m}\frac{(-\gamma_{12})^n}{n!} (-2Y^2)}{(-y+2Y)(y-4Y)^{g-1}\left(1+\frac{h}{y-4Y}\right)^g\exp(\frac{\hat{\gamma}}{y-4Y}) (-4Y^2)}\eeq

By expanding as in \cref{S5.2.1}, we find that \cref{e35} is 
\beq\label{e37}
2^{-2g}\sum_{r+l=2m-g+1,l\ge g-2}C_{r,0,l}\frac{(-1)^{l+n-1}}{(l-g+2)!n!}\sum_{k=0}^{\lfloor \frac{2g-n}{2}\rfloor} \frac{(-1)^k(2g-2k)!(2g-n)!g!}{(2g-2k-n)!k!(g-k)!}\eeq
\[+2^{-n}\sum_{r+s+l=2m-g+1,s\le g} C_{r,s,l}\frac{(-1)^{l+n}}{(s+l-2g+2)!n!}\frac{(2g-s-1)!g}{(g-s)!}
\]
where $C_{r,s,l}$ are defined by 
\[ \sum_{r,s,l\ge 0}C_{r,s,l}z^rx^st^l=\left( (1-2t)\sum_{k+l\le g-1}\frac{(-1)^{g-1-k-l}}{l!}\binom{g}{k} x^kz^l \right)^{-1}.\]

\begin{example}
When $g=m=n=3$, the wall crossing term from the second blowup is 
\beq\label{e38} 
-\frac{249}{2^{13}}\int_{\widetilde{\Gamma}}\gamma_{12}^3h^3 -\frac{19}{2^{14}}\int_{\widetilde{\Gamma}}\gamma_{12}^3\hat{\gamma}^3=-\frac{9}{8}.
\eeq 
\end{example}

\subsubsection{Intersection pairings on the partial desingularization}
Adding \cref{e30}, \cref{e33} and \cref{e37}, we obtain the following.
\begin{proposition}
$$\int_{\wcM^{\mathrm{ss}}_{2,0}} a_2^m\frac{f_2^n}{n!}|_{\wcM^{\mathrm{ss}}_{2,0}}= \frac{(-1)^{g-m}}{2^{2m-g+1}}\frac{B_{2g-2-2m}(1)}{(2g-2-2m)!} $$
$$+\delta_{g,n}2^{-3g+1}\sum_{k+l=g-1}2^{-l}\binom{g+k}{k}\binom{g+l-1}{l} $$
\[+2^{-2g}\sum_{r+l=2m-g+1,l\ge g-2}C_{r,0,l}\frac{(-1)^{l+n-1}}{(l-g+2)!n!}\sum_{k=0}^{\lfloor \frac{2g-n}{2}\rfloor} \frac{(-1)^k(2g-2k)!(2g-n)!g!}{(2g-2k-n)!k!(g-k)!}\]
\[+2^{-n}\sum_{r+s+l=2m-g+1,s\le g} C_{r,s,l}\frac{(-1)^{l+n}}{(s+l-2g+2)!n!}\frac{(2g-s-1)!g}{(g-s)!}.
\]
\end{proposition}

In particular, when $g=m=n=3$, we find that 
$$\int_{\wcM^{\mathrm{ss}}_{2,0}}a_2^3f_2^3|_{\wcM^{\mathrm{ss}}_{2,0}}=-\frac{1623}{256}\ne 0$$
and the JKKW pairings \cref{e31} are different from the JK pairings and the Joyce pairings.

\phantomsection
\addcontentsline{toc}{section}{References}
\sloppy
\setstretch{1.1}
\renewcommand*{\bibfont}{\normalfont\small}
\printbibliography

@article{moreira-parabolic,
    author       = {Moreira, Miguel},
    title        = {On the intersection theory of moduli spaces of parabolic bundles},
    version      = {preprint v1},
    eprinttype   = {arxiv},
    eprint       = {2503.08898v1},
    date         = {2025-03-11},
}

@article{bu-2023-curves,
    ids          = {Bu},
    author       = {Bu, Chenjing},
    title        = {Counting sheaves on curves},
    volume       = {434},
    doi          = {10.1016/j.aim.2023.109334},
    pages        = {109334, 87~pp.},
    journaltitle = {Adv. Math.},
    eprinttype   = {arxiv},
    eprint       = {2208.00927},
    year         = {2023},
}

@article{joyce-2008-configurations-iv,
    author       = {Joyce, Dominic},
    title        = {Configurations in abelian categories. IV},
    subtitle     = {Invariants and changing stability conditions},
    journaltitle = {Adv. Math.},
    volume       = {217},
    number       = {1},
    pages        = {125-204},
    doi          = {10.1016/j.aim.2007.06.011},
    year         = {2008},
    eprint       = {math/0410268},
    eprinttype   = {arxiv},
}

@article{joyce-wall-crossing,
    ids          = {Joy},
    author       = {Joyce, Dominic},
    title        = {Enumerative invariants and wall-crossing formulae in abelian categories},
    version      = {preprint v1},
    eprinttype   = {arxiv},
    eprint       = {2111.04694v1},
    date         = {2021-11-08},
}

@article{joyce-hall,
    author       = {Joyce, Dominic},
    title        = {Ringel–Hall style vertex algebra and Lie algebra structures on the homology of moduli spaces},
    url          = {https://people.maths.ox.ac.uk/joyce/hall.pdf},
    date         = {2018-03},
    note         = {preliminary version},
}

@article{gross-joyce-tanaka-2022,
    author       = {Gross, Jacob and Joyce, Dominic and Tanaka, Yuuji},
    title        = {Universal structures in $\mathbb{C}$-linear enumerative invariant theories},
    volume       = {18},
    pages        = {068, 61 pp.},
    doi          = {10.3842/sigma.2022.068},
    journaltitle = {{SIGMA}, Symmetry Integrability Geom. Methods Appl.},
    eprinttype   = {arxiv},
    eprint       = {2005.05637},
    year         = {2022},
}

@article{bu-vertex,
    author       = {Bu, Chenjing},
    title        = {Modules and generalizations of Joyce vertex algebras},
    version      = {preprint v1},
    eprinttype   = {arxiv},
    eprint       = {2506.00289v1},
    date         = {2025-05-30},
}

@article{atiyah-bott-1983,
    ids          = {AB},
    author       = {Atiyah, Michael F. and Bott, Raoul},
    title        = {The Yang–Mills equations over Riemann surfaces},
    volume       = {308},
    doi          = {10.1098/rsta.1983.0017},
    pages        = {523--615},
    journaltitle = {Philos. Trans. R. Soc. Lond. A},
    year         = {1983},
}

@article{jeffrey-1994-extended,
    ids          = {Jeff},
    author       = {Jeffrey, Lisa C.},
    title        = {Extended moduli spaces of flat connections on Riemann surfaces},
    volume       = {298},
    doi          = {10.1007/bf01459756},
    pages        = {667--692},
    number       = {4},
    journaltitle = {Math. Ann.},
    year         = {1994},
}

@article{jeffrey-kirwan-1998-curves,
    ids          = {JK},
    author       = {Jeffrey, Lisa C. and Kirwan, Frances C.},
    title        = {Intersection theory on moduli spaces of holomorphic bundles of arbitrary rank on a Riemann surface},
    volume       = {148},
    doi          = {10.2307/120993},
    pages        = {109--196},
    number       = {1},
    journaltitle = {Ann. Math.},
    eprinttype   = {arxiv},
    eprint       = {alg-geom/9608029},
    year         = {1998},
}

@article{jeffrey-kiem-kirwan-woolf-2003-cohomology,
    ids          = {JKKW1},
    author       = {Jeffrey, Lisa C. and Kiem, Young-Hoon and Kirwan, Frances C. and Woolf, Jonathan},
    title        = {Cohomology pairings on singular quotients in geometric invariant theory},
    volume       = {8},
    doi          = {10.1007/s00031-003-0510-y},
    pages        = {217--259},
    number       = {3},
    journaltitle = {Transform. Groups},
    eprinttype   = {arxiv},
    eprint       = {math/0101079},
    year         = {2003},
}

@article{jeffrey-kiem-kirwan-woolf-2006-intersection,
    ids          = {JKKW2},
    author       = {Jeffrey, Lisa C. and Kiem, Young-Hoon and Kirwan, Frances C. and Woolf, Jonathan},
    title        = {Intersection pairings on singular moduli spaces of bundles over a Riemann surface and their partial desingularisations},
    volume       = {11},
    doi          = {10.1007/s00031-005-1118-1},
    pages        = {439--494},
    number       = {3},
    journaltitle = {Transform. Groups},
    eprinttype   = {arxiv},
    eprint       = {math/0505362},
    year         = {2006},
}

@book{joyce-song-2012,
    ids          = {JS},
    author       = {Joyce, Dominic and Song, Yinan},
    title        = {A theory of generalized Donaldson–Thomas invariants},
    series       = {Mem. Am. Math. Soc.},
    number       = {1020},
    publisher    = {American Mathematical Society},
    eprinttype   = {arxiv},
    eprint       = {0810.5645},
    year         = {2012},
    doi          = {10.1090/s0065-9266-2011-00630-1},
}

@article{kiem-2004,
    ids          = {Kie1},
    author       = {Kiem, Young-Hoon},
    title        = {Intersection cohomology of quotients of nonsingular varieties},
    volume       = {155},
    doi          = {10.1007/s00222-003-0317-4},
    pages        = {163--202},
    number       = {1},
    journaltitle = {Invent. Math.},
    eprinttype   = {arxiv},
    eprint       = {math/0101254},
    year         = {2004},
}

@article{kiem-li-savvas-2024,
    ids          = {KLS},
    author       = {Kiem, Young-Hoon and Li, Jun and Savvas, Michail},
    title        = {Generalized Donaldson–Thomas invariants via Kirwan blowups},
    volume       = {127},
    doi          = {10.4310/jdg/1721071499},
    pages        = {1149--1205},
    number       = {3},
    journaltitle = {J. Differ. Geom.},
    eprinttype   = {arxiv},
    eprint       = {1712.02544},
    year         = {2024},
}

@article{kirwan-1985-partial,
    ids          = {Kir85},
    author       = {Kirwan, Frances Clare},
    title        = {Partial desingularisations of quotients of nonsingular varieties and their Betti numbers},
    volume       = {122},
    doi          = {10.2307/1971369},
    pages        = {41--85},
    number       = {1},
    journaltitle = {Ann. Math.},
    year         = {1985},
}

@thesis{martin-1997-thesis,
    ids          = {Mar},
    author       = {Martin, S. K.},
    title        = {Symplectic geometry and gauge theory},
    institution  = {University of Oxford},
    type         = {phdthesis},
    year         = {1997},
}

@article{mehta-seshadri-1980-moduli,
    ids          = {MeSe},
    author       = {Mehta, V. B. and Seshadri, C. S.},
    title        = {Moduli of vector bundles on curves with parabolic structures},
    volume       = {248},
    doi          = {10.1007/bf01420526},
    pages        = {205--239},
    journaltitle = {Math. Ann.},
    year         = {1980},
}

@book{mochizuki-2009-donaldson,
    ids          = {Moch},
    author       = {Mochizuki, Takuro},
    title        = {Donaldson type invariants for algebraic surfaces},
    series       = {Lect. Notes Math.},
    number       = {1972},
    publisher    = {Springer},
    year         = {2009},
    doi          = {10.1007/978-3-540-93913-9},
}

@book{mumford-fogarty-kirwan-1994,
    ids          = {MFK},
    author       = {Mumford, D. and Fogarty, J. and Kirwan, F.},
    title        = {Geometric invariant theory},
    edition      = {3},
    series       = {Ergeb. Math. Grenzgeb.},
    number       = {34},
    publisher    = {Springer-Verlag},
    year         = {1994},
}

@article{seshadri-1977,
    ids          = {Ses},
    author       = {Seshadri, C. S.},
    title        = {Moduli of vector bundles on curves with parabolic structures},
    volume       = {83},
    doi          = {10.1090/s0002-9904-1977-14210-9},
    pages        = {124--126},
    journaltitle = {Bull. Am. Math. Soc.},
    year         = {1977},
}

@article{intrinsic-dt-i,
    author       = {Bu, Chenjing and Halpern-Leistner, Daniel and Ibáñez Núñez, Andrés and Kinjo, Tasuki},
    title        = {Intrinsic Donaldson–Thomas theory. I. Component lattices of stacks},
    version      = {preprint v2},
    eprinttype   = {arxiv},
    eprint       = {2502.13892v2},
    date         = {2025-09-11},
}

@article{intrinsic-dt-ii,
    author       = {Bu, Chenjing and Ibáñez Núñez, Andrés and Kinjo, Tasuki},
    title        = {Intrinsic Donaldson–Thomas theory. {II}. Stability measures and invariants},
    version      = {preprint v1},
    eprinttype   = {arxiv},
    eprint       = {2502.20515v1},
    date         = {2025-02-27},
}

@article{intrinsic-dt-iii,
    author       = {Bu, Chenjing and Ibáñez Núñez, Andrés and Kinjo, Tasuki},
    title        = {Intrinsic Donaldson–Thomas theory. {III}. Wall-crossing and applications},
    note         = {in preparation},
}

@article{simpson-1994,
    ids          = {Simpson},
    author       = {Simpson, Carlos T.},
    title        = {Moduli of representations of the fundamental group of a smooth projective variety. I},
    volume       = {79},
    doi          = {10.1007/bf02698887},
    pages        = {47--129},
    journaltitle = {Publ. Math., Inst. Hautes Étud. Sci.},
    year         = {1994},
}

@article{seidel-thomas-2001,
    author       = {Seidel, Paul and Thomas, Richard},
    title        = {Braid group actions on derived categories of coherent sheaves},
    volume       = {108},
    doi          = {10.1215/s0012-7094-01-10812-0},
    pages        = {37--108},
    number       = {1},
    journaltitle = {Duke Math. J.},
    eprinttype   = {arxiv},
    eprint       = {math/0001043},
    year         = {2001},
}

\authorinforule

\authorinfo{Chenjing Bu}
    {bucj@mailbox.org}
    {Mathematical Institute, University of Oxford, Oxford OX2 6GG, United Kingdom}

\authorinfo{Young-Hoon Kiem}
    {kiem@kias.re.kr}
    {School of Mathematics, Korea Institute for Advanced Study, 85 Hoegiro, Dongdaemun-gu, Seoul~02455, Korea}

\end{document}